\documentclass[11pt,reqno]{amsart}
\usepackage{mathrsfs}
\usepackage{url}
\usepackage{mathtools}
\usepackage{latexsym,epsfig,amssymb,amsmath,amsthm,color,url,bm}
\usepackage[inline,shortlabels]{enumitem}
\usepackage{hyperref}
\usepackage[foot]{amsaddr}
\usepackage{amsmath,amsbsy}
\usepackage{mwe}
\RequirePackage[numbers]{natbib}
\usepackage{mathptmx}
\usepackage[text={16cm,24cm}]{geometry}
\usepackage{listings}

\allowdisplaybreaks 
 \numberwithin{equation}{section}
\newtheorem{theorem}{Theorem}[section]
\newtheorem{prop}[theorem]{Proposition}
\newtheorem{lemma}[theorem]{Lemma}
\newtheorem{corollary}[theorem]{Corollary}

\newcommand\Item[1][]{%
  \ifx\relax#1\relax  \item \else \item[#1] \fi
  \abovedisplayskip=0pt\abovedisplayshortskip=0pt~\vspace*{-\baselineskip}}

\theoremstyle{definition}

\theoremstyle{definition}

\DeclareMathOperator{\Prob}{\mathbf{P}}
\DeclareMathOperator{\E}{\mathbf{E}}

\title[]{Elephant random walks with multiple extractions and general reinforcement functions}
\date{}
\author{Moumanti Podder and Archi Roy}
\address{Moumanti Podder, Indian Institute of Science Education and Research (IISER) Pune, Dr.\ Homi Bhabha Road, Pashan, Pune 411008, Maharashtra, India.}
\address{Archi Roy, Indian Institute of Management, Kozhikode 673570, Kerala, India}
\email{moumanti@iiserpune.ac.in}
\email{archiroy@iimk.ac.in}

\begin{document}
\bibliographystyle{plainnat}

\begin{abstract}
We consider a generalized model of elephant random walks wherein the walker, during the $(n+1)$-st time-stamp, draws from the past (i.e.\ the set $\{1,2,\ldots,n\}$) a sample of $k$ time-stamps, either with replacement or without, where $k$ may either remain fixed as $n$ grows, or $k=k(n)$ may grow with $n$. Letting $U_{n,1}, U_{n,2}, \ldots, U_{n,k}$ denote the time-stamps sampled, the step taken by the walker during the $(n+1)$-st time-stamp, denoted $X_{n+1}$, is a $\pm 1$-valued random variable whose distribution depends on the proportion of $(+1)$-valued steps out of $X_{U_{n,1}},X_{U_{n,2}},\ldots,X_{U_{n,k}}$ via a reinforcement function $f$. In this paper, we investigate the asymptotic behaviour -- i.e.\ strong and weak convergence -- of this random walk model under suitable assumptions made on the function $f$ (as well as on the sequence $\{k(n)\}$ when the sample size varies with $n$).

\end{abstract}



\maketitle

\tableofcontents
\section{Introduction}\label{sec:intro}
Motivated by the theory of increasing returns in economics (see, for instance, \citep{arthur1994increasing}), we consider a situation where two companies release similar products into the market around the same time, and we propose and analyze a model that tracks the sale, over time, of one of these products relative to the other. Our model involves a setting in which consumers possess imperfect information (for instance, see \citep{arthur1yu}) about the underlying quality of the competing products, and therefore, rely on informal signals such as peer choices, reviews, or aggregated opinions, to guide their decisions. These feedback mechanisms introduce path dependence into the market dynamics, often ultimately leading to the dominance of one product.

We begin with a brief and informal overview of this model, which is formally introduced in \S\ref{sec:model}. In this work, time is considered discrete, and hence indexed by the set $\mathbb{N}_{0}$ of non-negative integers. Each of the two competing products is assumed to have been released into the market at time $t=0$, and $t=n$ indicates the moment at which the $n$-th sale happens (i.e.\ combining the sales of both types of products, a total of $n$ items have been sold in the time-interval $[0,n]$). A simplifying assumption we make here is that each customer purchases precisely one product at a time. The $(n+1)$-st customer samples $k$ customers, with or without replacement, from the past (i.e.\ from among the customers who made their purchases at time-stamps $t=1$, $t=2$, $\ldots$, $t=n$), and collects their responses (regarding the choices these customers made). The decision made by the $(n+1)$-st customer is a random variable whose distribution is \emph{some} function of the responses obtained from the sample drawn. The sample size, $k$, can either remain constant as $n$ varies, or be a function of $n$ (in which case we write $k=k(n)$) that grows with $n$. Following the terminology of \cite{franchini2025elephant}, we refer to this framework as an elephant random walk model with \emph{multiple extractions}.

Although our work is largely motivated by the market share example, the underlying structure has broad applicability across a range of fields \citep{khanin2001probabilistic,oliveira2008balls}. Here, we highlight that a similar model has been considered in \cite{franchini2023large} and \cite{franchini2025elephant}, but with the sample size being an odd positive integer that remains fixed as $n$ varies, and the reinforcement function $f$ (see \S\ref{sec:model}, and in particular, \eqref{X_{n+1}_distribution}, for understanding the role of the reinforcement function in our model) being the majority function (in other words, $f(x)=1$ for $x>1/2$, and $f(x)=0$ for $x<1/2$). In this context, we also mention \cite{franchini2017large} that considers a two-colour generalized urn model wherein the urn configuration is updated at each time-stamp via a stochastic update rule that depends on the current configuration of the urn via a continuous reinforcement function. In \cite{franchini2023large} and \cite{franchini2017large}, the primary focus is on strong convergence and large deviation principles. 

Denoting the product manufactured by the first company by $P_{1}$, and the second by $P_{2}$, we maintain a chart that tracks the sale of each of $P_{1}$ and $P_{2}$ over time. Every time a unit of $P_{1}$ is purchased, we mark this sale by adding a $(+1)$ to our chart, while every time a unit of $P_{2}$ is purchased, we mark it using a $(-1)$, so that the sum of all the markings up to and including time-stamp $n$ equals the \emph{relative performance} of $P_{1}$ with respect to $P_{2}$ up to and including time-stamp $n$ (i.e.\ to what extent the sale of $P_{1}$ exceeds, or trails behind, that of $P_{2}$, provided a total of $n$ units, of $P_{1}$ and $P_{2}$ combined, have been sold so far). Such a chart is evidently reminiscent of a random walk in which the position of the walker at time $n$ is given by the relative performance of $P_{1}$ with respect to $P_{2}$ up to and including time-stamp $n$. This observation, in turn, opens up another perspective from which our model may be studied and may find applicability: that of the relatively novel notion of \emph{elephant random walks}.

A third perspective from which our model can be viewed is that of a \emph{generalized urn model} comprising balls of two colours, namely, $(+1)$ and $(-1)$. At time $t=0$, we begin with an empty urn, and at time $t=1$, we add to it a ball which is of colour $(+1)$ with a certain prespecified probability $q$, and of colour $(-1)$ with the remaining probability $(1-q)$. At time-stamp $t=(n+1)$, $k$ balls are chosen, uniformly at random, either with replacement or without, from the existing balls in the urn, and their colours noted. The sampled balls are then returned to the urn, along with a new ball whose colour is a $\pm 1$-valued random variable with probability distribution depending on the colours of the sampled balls.    

The primary objective of this work is to understand the limiting behaviour of the stochastic process outlined above. In particular, when the model is studied as a version of the celebrated elephant random walk, with the position of the walker (or elephant) at time $n$ denoted by $S_{n}$, we analyze the almost sure convergence of the sequence $\{S_{n}/n\}$, as well as the distributional convergence of a sequence obtained via a suitable scaling of $S_{n}$, as $n$ approaches $\infty$. Such results on strong and weak convergence are achieved via the representation of our stochastic process as a suitable \emph{stochastic approximation process}. In fact, results from the existing literature on stochastic approximation are instrumental in proving the results in this paper.

\subsection{Organization of the paper} The rest of the paper has been laid out as follows. The model has been formally introduced in \S\ref{sec:model}: the model obtained when the sampling scheme involved is \emph{with replacement} has been outlined in \S\ref{subsec:with}, and the model obtained when the sample drawn is \emph{without replacement} has been outlined in \S\ref{subsec:without}. A detailed discussion of the existing literature, that is pertinent to the work done in this paper, has been provided in \S\ref{sec:lit_review}. The main results of the paper have been stated in \S\ref{sec:main_results}, with \S\ref{subsec:results_with} addressing the model described in \S\ref{subsec:with}, and \S\ref{subsec:without_results} addressing the model described in \S\ref{subsec:without}. All of the results in this paper have been proved, in detail, in \S\ref{sec:all_proofs} -- in particular, the proofs of the results from \S\ref{subsec:results_with} have been included in \S\ref{subsec:proofs_with}, whereas the proofs of the results from \S\ref{subsec:without_results} have been included in \S\ref{subsec:proofs_without}.

\section{Formal description of the model}\label{sec:model}
\subsection{Model involving sampling with replacement}\label{subsec:with} We begin by describing the model in which the $(n+1)$-st customer (equivalently, as the elephant is about to decide the step it would take at the $(n+1)$-st time-stamp, in our version of the elephant random walk) samples, \emph{with} replacement, $k$ customers from the past (equivalently, the elephant samples, \emph{with} replacement, $k$ of the steps that it has taken so far). At the very outset, we introduce the notation $[n]$ to indicate the set $\{1,2,\ldots,n\}$ of the first $n$ positive integers, for each $n \in \mathbb{N}$. We set $X_{0}=S_{0}=0$, and $X_{1}$ to be a Rademacher$(q)$ random variable for a prespecified $q \in [0,1]$, i.e.\
\[
 X_{1} = 
  \begin{cases} 
   1 & \text{with probability } q, \\
   -1 & \text{with probability } (1-q).
  \end{cases}
\]
For each $n \in \mathbb{N}$, we define $U_{n,i}$, for $i \in [k]$, to be independent and identically distributed random variables, with each $U_{n,i}$ distributed uniformly over the support $[n]$. Evidently, the set $\{U_{n,1},U_{n,2},\ldots,U_{n,k}\}$ can be interpreted as a sample, drawn with replacement, from among the first $n$ customers (equivalently, it is a set of time-stamps, drawn with replacement, from $[n]$, which is the past up to and including time-stamp $n$). We let $\mathcal{F}_{n}$ indicate the $\sigma$-field consisting of all information pertaining to our stochastic process up to and including time-stamp $n$. Conditioned on $\mathcal{F}_{n}$ and the value of $U_{n,i}$ for each $i \in [k]$, we define
\begin{equation}\label{X_{n+1}_distribution}
X_{n+1}=  
  \begin{cases} 
   1 & \text{with probability } pf\left(C_{n}^{+}/k\right)+(1-p)\left\{1-f\left(C_{n}^{+}/k\right)\right\}, \\
   -1 & \text{with probability } (1-p)f\left(C_{n}^{+}/k\right)+p\left\{1-f\left(C_{n}^{+}/k\right)\right\},
  \end{cases}
\end{equation}
where $p \in (0,1)\setminus\{1/2\}$ is a prespecified parameter that is often referred to as the \emph{memory parameter} or \emph{self-excitation parameter} in the literature, $f:[0,1] \rightarrow [0,1]$ is a predefined function that may be referred to as a \emph{reinforcement function} or a \emph{replacement function} (see \cite{franchini2017large,franchini2023large,ruszel2024positive} among others to find different examples of function classes to which $f$ could belong) when the stochastic process is viewed as a generalized urn model, and (henceforth, for any event $A$, letting $\chi\{A\}$ denote the indicator random variable for $A$) 
\begin{equation}
C_{n}^{+}=\sum_{i=1}^{k}\chi\left\{X_{U_{n,i}}=1\right\}=\frac{1}{2}\sum_{i=1}^{k}\left(X_{U_{n,i}}+1\right)\label{C_{n}^{+}}
\end{equation}
is the number of occurrences of $+1$ in the sample drawn (in the context of customers making decisions based on responses sampled from past customers, $C_{n}^{+}$ equals the number of customers sampled, up to and including time-stamp $n$, who chose product $P_{1}$, whereas when the model is studied as an elephant random walk, $C_{n}^{+}$ equals the number of $(+1)$-valued steps that the elephant samples from the past, up to and including time-stampe $n$). We set
\begin{equation}
S_{n}=\sum_{i=1}^{n}X_{i},\label{S_{n}}
\end{equation}
so that $S_{n}$ denotes the position of the elephant (i.e.\ its displacement from the origin, which is where the walk begins from) at time-stamp $n$, or, equivalently, the relative performance of product $P_{1}$, in comparison to product $P_{2}$, up to and including time-stamp $n$. Note that it makes sense to exclude the case of $p=1/2$, since this value of $p$ boils our stochastic process down to the ordinary simple symmetric random walk on $\mathbb{Z}$, and this has already been thoroughly studied in the literature.

The description of this stochastic process remains the same even when we replace $k$ by $k(n)$, a function that maps the set $\mathbb{N}$ of all positive integers to itself, with the restriction that $k(n) \leqslant n$. As mentioned in \S\ref{sec:intro}, we are primarily interested in $k(n)$ that grows with $n$ as $n$ approaches $\infty$.

\subsection{Model involving sampling without replacement}\label{subsec:without} The description of the model remains nearly the same as we consider a sampling scheme that is without replacement. Here, we begin with $X_{0}=0$, and random variables $X_{1}, X_{2}, \ldots, X_{k}$ that are independent and identically distributed as Rademacher$(q)$. For each $n \geqslant k$, at time-stamp $(n+1)$, a sample $(U_{n,1}, U_{n,2}, \ldots, U_{n,k})$, is drawn, without replacement, from the set $[n]$ of past indices (or time-stamps). Evidently, $\left(U_{n,1}, U_{n,2}, \ldots, U_{n,k}\right)$ denotes a $k$-dimensional random ordered tuple whose support is given by the set of all ordered $k$-tuples in which all the coordinates are distinct and each coordinate belongs to $[n]$, such that
\begin{equation}
\Prob\left[\left(U_{n,1}, U_{n,2}, \ldots, U_{n,k}\right)=(i_{1},i_{2},\ldots,i_{k})\right]=\left\{{n \choose k}k!\right\}^{-1} \text{ for all distinct indices } i_{1},i_{2},\ldots,i_{k}\in [n].\label{without_replacement_distribution}
\end{equation}
We define $C_{n}^{+}$, conditioned on $\mathcal{F}_{n}$, to be as in \eqref{C_{n}^{+}}, and define the distribution of $X_{n+1}$, conditioned on $\mathcal{F}_{n}$ and the realized value of $(U_{n,1}, U_{n,2}, \ldots, U_{n,k})$, as in \eqref{X_{n+1}_distribution}. Finally, we define $S_{n}$ to be as in \eqref{S_{n}}.

Once again, we consider the version of this model where $k=k(n)$ is a function of $n$ (and that grows with $n$ as $n$ approaches $\infty$, with the restriction that $k(n) \leqslant n$ for each $n$). We assume that $k(n)$ is a monotonically increasing function of $n$, and we write $k(n)\uparrow\infty$ to indicate that the sequence $\{k(n)\}$ approaches $\infty$ as $n\rightarrow\infty$. The assumption that $k(n)\leqslant n$, for all $n\in\mathbb{N}$, allows us to let $X_{1}$ follow Rademacher$(q)$, and the distribution of $X_{n+1}$, for each $n\geqslant 1$, to be exactly as in \eqref{X_{n+1}_distribution}, where $C_{n}^{+}$, conditioned on $\mathcal{F}_{n}$, is as defined in \eqref{C_{n}^{+}}.

\section{Review of literature pertinent to the work in this paper}\label{sec:lit_review}
We now turn to a brief review of the key developments in the literature related to the elephant random walk (ERW) model. The classical ERW model was proposed by \cite{schutz2004elephants} with the motivation of introducing long memory into random walk models and studying how, in such cases, the asymptotic behavior differs from the more traditional simple random walk models where the walker takes each step independent of the past. In \cite{schutz2004elephants}, a one-dimensional random walk $\{S_n\}_{n\in\mathbb{Z}_+}$ on the $\mathbb{Z}$-lattice was considered, wherein the walker starts from a position $S_0$ (the general convention is to take $S_0=0$) at time-stamp $t=0$, and during the first time-stamp, the walker takes a forward (i.e.\ $+1$) step with probability $q$ and a backward (i.e.\ $-1$) step with probability $(1-q)$, where $q\in (0,1)$ is prespecified. At all subsequent time-stamps, the walker keeps a complete memory of all of the steps taken so far. At time-stamp $(n+1)$, for each $n \in \mathbb{N}$, the walker takes a random step of unit size, either forward or backward, so that $S_{n+1}=S_n+X_{n+1}$, where $\{X_n: n \in \mathbb{N}, n \geqslant 2\}$ is a sequence of $(\pm 1)$-valued random variables defined recursively as follows:
\begin{equation*}
    X_{n+1}=\begin{cases}
        X_{U_n}\text{ with probability }p,\\
        -X_{U_n}\text{ with probability }(1-p),
    \end{cases}
\end{equation*}
where $U_n$ is a randomly chosen time-stamp from the past, i.e.\ $[n]$, and $p\in (0,1)$ is referred to, in various contexts, as the \textit{trust parameter}, the \emph{memory parameter} or the \emph{self-excitation parameter}, and its value may differ from that of $q$. Several extensions and variations of the classical ERW have since been proposed. For instance, \cite{gut2021variations} introduced a framework where the walker remembers only a few of the most recent steps, certain distant steps from the past, or a combination of both. In a related development, \cite{gut2022elephant} proposed a model where the memory of the walker grows over time. A ``lazy” version of the walk, wherein the walker may, during any time-stamp, choose to stay in the same position as they were in during the previous time-stamp, instead of moving forward or backward, was proposed by \cite{gut2021number}. Random step sizes were incorporated into the model by \cite{dedecker2023rates} and \cite{fan2024cramer}, while \cite{dhillon2025elephantrandomwalkrandom} recently introduced a version in which the memory itself is random. Reinforced variations of the walk have been explored in works such as \cite{laulin2022new} and \cite{laulin2022introducing}. Multidimensional extensions of the elephant random walk have also received considerable attention \citep[see, e.g.,][]{bercu2025multidimensionalelephantrandomwalk, bertenghi2022functional, marquioni2019multidimensional,maulik2024asymptotic}. Other significant variants of the walk can be found in works such as \citep{bercu2022elephant,harbola2014memory,harris2015random,kim2014anomalous,kursten2016random,roy2024elephant,nakano2025elephant,nakano2025limit, serva2013scaling}. More recently, \cite{majumdar2025limit} investigated limit theorems for a broad class of step-reinforced random walks that includes the classical ERW as a special case. Motivated by earlier studies such as \cite{kursten2016random} and \cite{kim2014anomalous}, they analyzed a model in which the walker either repeats a previously taken step or takes a new step with complementary probability. 

The asymptotic behavior of the ERW is well investigated in the literature \citep{boyer2014solvable, coletti2017central, cressoni2013exact, kumar2010memory}. In this context, \cite{da2013non} proposed strong numerical evidence that the displacement distribution of ERWs in some regimes of the trust parameter $p$ is non-Gaussian as well as non-L\'evy. In an important contribution, \cite{bercu2017martingale} and \cite{laulin2022new} proposed a martingale approach in deriving the asymptotics, where results on weak convergence and large deviation are derived using the martingale central limit theorem and the law of iterated logarithms respectively. The martingale approach was also used by papers such as \cite{bercu2021center}, \cite{coletti2017central} and \cite{roy2024phase}. A powerful set of analytical tools was introduced in \cite{baur2016elephant}, which studied the asymptotic behavior of the process across various regimes by reformulating the walk as a Pólya urn scheme \citep[see][for a detailed discussion]{janson2004functional}.  In this paper, we take inspiration from existing works such as \cite{das2024elephant}, \cite{maulik2024asymptotic} and \cite{zhang2024stochastic}, to derive the asymptotic results using a stochastic approximation \citep{duflo2013random} technique. Stochastic approximation (SA), originally introduced by \cite{robbins1951stochastic} to find the root of a regression function via successive approximations, has since found wide application across various scientific domains (see \cite{kushner2010stochastic} and \cite{lai2003stochastic} for reviews). While the classical SA theory focused on almost sure convergence to a unique limit, it was later extended to address weak convergence results as well, particularly in control and optimization theory \citep{asi2019stochastic, dupuis1985stochastic, dupuis1989stochastic, zhang2024constant}. To the best of our knowledge, \cite{gangopadhyay2019stochastic} was the first to apply SA methods to analyze urn models with random replacement matrices, laying the groundwork for using SA in processes interpretable as generalized urn models, including the random walk model considered in this paper, as previously discussed. This line of work was further developed in \cite{gangopadhyay2022almost}, where the authors studied a variant of the ERW on the nonnegative integer lattice in multiple dimensions, with memory-driven reinforcement and potential delay at each step. Although their results focused on almost sure convergence, their approach motivated us to adopt SA as the central tool to analyze the convergence behavior of our proposed random walk model. A very important contribution to this literature has been recently accomplished by \cite{maulik2024asymptotic}, who developed the SA theory for multidimensional ERW.

\section{Main results of this paper}\label{sec:main_results}
We begin by enumerating the results pertaining to the model described in \S\ref{subsec:with} (i.e.\ where the sampling takes place with replacement) in \S\ref{subsec:results_with}. Following this, in \S\ref{subsec:without_results}, we state the results that are relevant for the model described in \S\ref{subsec:without} (i.e.\ where the sampling scheme is without replacement).
\subsection{When the sampling scheme is with replacement}\label{subsec:results_with} We begin by stating Theorem~\ref{thm:main_1} and Propositions~\ref{prop:main_1}, \ref{prop:main_2} and \ref{prop:main_3}, each of which is concerned with the almost sure convergence of the sequence $\{S_{n}/n\}$, where $S_{n}$ is as defined in \eqref{S_{n}}, as time $n$ approaches $\infty$, when a sample of constant size, $k$, is drawn with replacement from the past during each time-stamp:
\begin{theorem}\label{thm:main_1}
Let $\{S_{n}\}$ be the stochastic process defined as in \eqref{S_{n}}, where the underlying model is as described in \S\ref{subsec:with}, with $p\in(0,1)\setminus\{1/2\}$ and the sample size, $k$, remaining constant with $n$. Recall the function $f$ introduced via \eqref{X_{n+1}_distribution}, and let the function $g:[0,1]\rightarrow [0,1]$ be defined as
\begin{equation}\label{g_defn}
g(x)=pf(x)+(1-p)\{1-f(x)\} \text{ for each } x \in [0,1].
\end{equation}
Let us further define the function $H:[-1,1]\rightarrow [-1,1]$ as
\begin{equation}\label{H_defn}
H(x)=2^{1-k}\sum_{i=0}^{k}g\left(\frac{i}{k}\right){k \choose i}(1+x)^{i}(1-x)^{k-i}-1, \text{ for all } x \in [-1,1].
\end{equation}
If $H$ has a \emph{unique} fixed point, say $x^{*}$, in the interval $[-1,1]$, the sequence $\{S_{n}/n\}$ converges almost surely to $x^{*}$. In particular, $H$ has a unique fixed point, say $x^{*}$, in $[-1,1]$, and $\{S_{n}/n\}$ converges almost surely to $x^{*}$, if one of the following is true:
\begin{enumerate}[label=(B\arabic*), ref=B\arabic*]
\item \label{strict_decreasing} $H$ is strictly decreasing throughout $[-1,1]$;
\item \label{strict_convex} $H$ is either strictly convex throughout $[-1,1]$, or strictly concave throughout $[-1,1]$;
\item \label{contraction} $H$ is a contraction throughout $[-1,1]$. 
\end{enumerate} 
\end{theorem}

\begin{prop}\label{prop:main_1}
Consider the model in \S\ref{subsec:with} with the sample size, $k$, remaining constant with $n$. If 
\begin{enumerate}[label=(C\arabic*), ref=C\arabic*]
\item \label{C1} either $p\in(1/2,1)$ and $f(i/k)\leqslant f((i-1)/k)$ for each $i \in [k]$, with the inequality being strict for at least one $i \in [k]$,
\item \label{C2} or $p\in(0,1/2)$ holds and $f(i/k)\geqslant f((i-1)/k)$ for each $i \in [k]$, with the inequality being strict for at least one $i \in [k]$,
\end{enumerate}
then the sequence $\{S_{n}/n\}$ converges almost surely to some $x^{*} \in (-1,1)$.
\end{prop}
An example of a function that satisfies \eqref{C1} of Proposition~\ref{prop:main_1} is $f(x)=cx$ for all $x\in[0,1]$, where $0<c<p(2p-1)^{-1}$, whereas an example of a function that satisfies \eqref{C2} is $f(x)=c'(1-x)$ for all $x\in[0,1]$, where $0<c'<(1-p)(1-2p)^{-1}$.

\begin{prop}\label{prop:main_2}
In the model described in \S\ref{subsec:with}, with the sample size, $k$, remaining constant with $n$, if 
\begin{enumerate}[label=(D\arabic*), ref=D\arabic*]
\item \label{strict_convexity_f_cond} either $f((j+2)/k)+f(j/k)\geqslant 2f((j+1)/k)$ for each $j \in [k]\cup\{0\}$, with the inequality being strict for at least one $j \in [k]\cup\{0\}$, 
\item \label{strict_concavity_f_cond} or $f((j+2)/k)+f(j/k)\leqslant 2f((j+1)/k)$ for each $j \in [k]\cup\{0\}$, with the inequality being strict for at least one $j \in [k]\cup\{0\}$, 
\end{enumerate}
then the sequence $\{S_{n}/n\}$ converges almost surely to some $x^{*} \in (-1,1)$.
\end{prop}
An example of a function $f$ that satisfies the hypothesis of Proposition~\ref{prop:main_2} is $f(x)=c'' e^{x}$ for all $x\in[0,1]$, where $c''>0$ obeys either the inequality $c''<(1-p)(1-2p)^{-1}$, or the inequality $c''<p(2p-1)^{-1}e^{-1}$.

\begin{prop}\label{prop:main_3}
When the underlying model is as described in \S\ref{subsec:with}, with the sample size, $k$, remaining constant with $n$, the sequence $\{S_{n}/n\}$ converges almost surely to some $x^{*} \in (-1,1)$ whenever one of the following two conditions holds:
\begin{enumerate}[label=(E\arabic*), ref=E\arabic*]
\item $1/2<p<1/2+1/(2k)$ and $f(1)<p/(2p-1)$,
\item or $1/2-1/(2k)<p<1/2$ and $f(0)<(1-p)/(1-2p)$.
\end{enumerate}
\end{prop}

Theorem~\ref{thm:main_2} concerns itself with the distributional convergence of a suitably scaled version of the sequence $\{S_{n}\}$ when the underlying model is as described in \S\ref{subsec:with} and the sample size, $k$, remains constant as $n$ varies. For a given sequence $\{W_{n}\}$ of random variables, writing $W_{n}\xlongrightarrow{D}N(\mu,\sigma^{2})$, for any $\mu\in\mathbb{R}$ and $\sigma>0$, indicates that $\{W_{n}\}$ converges in distribution to a normal random variable with mean $\mu$ and variance $\sigma^{2}$.
\begin{theorem}\label{thm:main_2}
Consider the sequence $\{S_{n}\}$ arising from the model described in \S\ref{subsec:with}, where $k$ has been kept constant as $n$ varies, and $p\in(0,1)\setminus\{1/2\}$. If the function $H$, defined in \eqref{H_defn}, has a unique fixed point, say $x^{*}$, in $[-1,1]$, then, setting $\tau=1-H'(x^{*})>0$, we have
\begin{align}
{}&\sqrt{n}\left(S_{n}/n-x^{*}\right)\xlongrightarrow{\text{D}}N\left(0,\left(1-{x^{*}}^{2}\right)(2\tau-1)^{-1}\right) \quad \text{when} \quad \tau>1/2,\nonumber\\
{}& \sqrt{n/\ln n}\left(S_{n}/n-x^{*}\right)\xlongrightarrow{\text{D}}N\left(0,1-{x^{*}}^{2}\right) \quad \text{when} \quad \tau=1/2,\nonumber
\end{align}
and $n^{\tau}\left(S_{n}/n-x^{*}\right)$ converges almost surely to a finite random variable when $\tau<1/2$. If one of \eqref{strict_decreasing}, \eqref{strict_convex} and \eqref{contraction} holds, the conclusion drawn above remains true. Even more specifically, if the hypothesis stated in one of Propositions~\ref{prop:main_1}, \ref{prop:main_2} and \ref{prop:main_3} is true, then too, the conclusion drawn above is true. 
\end{theorem}

Henceforth, by the notation $f \in \mathcal{C}^{k}[0,1]$, for any function $f:[0,1]\rightarrow \mathbb{R}$ and any $k\in\mathbb{N}$, we mean that the $k$-th derivative of $f$, i.e.\ $f^{(k)}$, exists and is continuous throughout $[0,1]$ (when $k=1$, we write $f^{(1)}=f'$, and when $k=2$, we write $f^{(2)}=f''$); by $f \in \mathcal{C}^{0}[0,1]$ we mean that $f$ is continuous throughout $[0,1]$. The next two results, i.e.\ Theorems~\ref{thm:main_3} and \ref{thm:main_4}, concern themselves with the strong and weak convergence of our stochastic process when the sampling scheme is with replacement and the sample size, $k(n)$, grows with $n$.
\begin{theorem}\label{thm:main_3}
Consider the model described in \S\ref{subsec:with}, with $p\in(0,1)\setminus\{1/2\}$, where the sampling takes place with replacement, but this time, the sample size, $k=k(n)$, is increasing in $n$, which we, henceforth, write as $k(n)\uparrow \infty$ as $n\rightarrow \infty$. We assume that one of the following conditions holds:
\begin{enumerate}[label=(F\arabic*), ref=F\arabic*]
\item \label{D1} The function $f\in\mathcal{C}^{0}[0,1]$, and that the series
\begin{equation}
\sum_{n=1}^{\infty}\frac{1}{n+1}\sup\left\{\left|f(x)-f(y)\right|:x,y\in[0,1],|x-y|<k(n)^{-1/2}\right\} \text{ converges.}\label{series_convergence_criterion_1}
\end{equation}
\item \label{D2} The function $f$ is H\"{o}lder-continuous with exponent $\alpha$ and constant $L$, for some $0< \alpha \leqslant 1$, i.e.\ the function $f$ satisfies the inequality $\left|f(x)-f(y)\right|\leqslant L|x-y|^{\alpha}$ for all $x, y \in [0,1]$, and the series 
\begin{equation}
\sum_{n=1}^{\infty}(n+1)^{-1}k(n)^{-\alpha/2} \text{ converges.}\label{series_convergence_criterion_2}
\end{equation}
\item \label{D3} The function $f\in \mathcal{C}^{1}[0,1]$ and the series 
\begin{equation}
\sum_{n=1}^{\infty}\frac{1}{(n+1)\sqrt{k(n)}}\sup\left\{\left|f'(x)-f'(y)\right|:x,y\in[0,1],|x-y|<k(n)^{-1/2}\right\} \text{ converges}.\label{series_convergence_criterion_3}
\end{equation}
\item \label{D4} The function $f\in\mathcal{C}^{2}[0,1]$, and the series
\begin{equation}
\sum_{n=1}^{\infty}(n+1)^{-1}k(n)^{-1} \text{ converges}. \label{series_convergence_criterion_4}
\end{equation}
\end{enumerate}
If the function $g$, defined in \eqref{g_defn}, has a unique fixed point, $x^{*}$, in $[0,1]$, the sequence $\{n^{-1}S_{n}\}$ converges almost surely to $(2x^{*}-1)$. More specifically, this conclusion remains true if one of \eqref{D1}, \eqref{D2}, \eqref{D3} and \eqref{D4} holds, and one of the following is true:
\begin{enumerate}[label=(G\arabic*), ref=G\arabic*]
\item \label{strict_decreasing_1} the function $f\in\mathcal{C}^{1}[0,1]$, and either $p\in(1/2,1)$ and $f$ is strictly decreasing over $[0,1]$, or $p\in(0,1/2)$ and $f$ is strictly increasing over $[0,1]$,
\item \label{strict_convex_1} the function $f\in\mathcal{C}^{2}[0,1]$, and either $f$ is strictly convex throughout $(0,1)$, or $f$ is strictly concave throughout $(0,1)$,
\item \label{contraction_1} the function $f$ is Lipschitz, with $|f(x)-f(y)|\leqslant c|x-y|$ for all $x, y \in [0,1]$, such that $c|2p-1|<1$.
\end{enumerate}
\end{theorem}

\begin{theorem}\label{thm:main_4}
Consider the model described in \S\ref{subsec:with}, with $p\in(0,1)\setminus\{1/2\}$ and the sample size $k(n)\uparrow \infty$ as $n\rightarrow \infty$. We assume that the function $f$, introduced via \eqref{X_{n+1}_distribution}, is in $\mathcal{C}^{2}[0,1]$, and the function $g$, defined in \eqref{g_defn}, has a unique fixed point $x^{*}\in[0,1]$. Letting $\tau=1-g'(x^{*})>0$, we assume that 
\begin{equation}\label{series_convergence_criterion_5}
\text{when } \tau<1/2:
\begin{cases}
&\sum_{i=1}^{n}(i+1)^{\tau-1}k(i)^{-1} \text{ converges as } n\rightarrow \infty,\\
&\sum_{i=1}^{n}(i+1)^{\tau-1}\sum_{t=1}^{i-1}(t+1)^{-1}k(t)^{-1} \text{ converges as } n\rightarrow \infty,
\end{cases}
\end{equation}
\begin{equation}\label{series_convergence_criterion_6}
\text{when } \tau=1/2:
\begin{cases}
&(\ln n)^{-1/2}\sum_{i=1}^{n}(i+1)^{-1/2}k(i)^{-1} \rightarrow 0 \text{ as } n \rightarrow \infty,\\
&(\ln n)^{-1/2}\sum_{i=1}^{n}(i+1)^{-1/2}\sum_{t=1}^{i-1}(t+1)^{-1}k(t)^{-1} \rightarrow 0 \text{ as } n \rightarrow \infty,
\end{cases}
\end{equation}
\begin{equation}\label{series_convergence_criterion_7}
\text{when } \tau>1/2:
\begin{cases}
&n^{1/2-\tau}\sum_{i=1}^{n}(i+1)^{\tau-1}k(i)^{-1} \rightarrow 0 \text{ as } n \rightarrow \infty,\\
&n^{1/2-\tau}\sum_{i=1}^{n}(i+1)^{\tau-1}\sum_{t=1}^{i-1}(t+1)^{-1}k(t)^{-1} \rightarrow 0 \text{ as } n \rightarrow \infty.
\end{cases}
\end{equation}
In addition, we assume that \eqref{D4} is true for all values of $\tau$. Then we have
\begin{equation}\label{eq:distributional_convergence_with_replacement_growing_size}
\begin{cases}
{}&n^{\tau}\left\{S_{n}/n-(2x^{*}-1)\right\} \text{ converges to a finite random variable almost surely when } \tau<1/2,\\
{}&\sqrt{n/\ln n}\left\{S_{n}/n-(2x^{*}-1)\right\} \xlongrightarrow{D} N\left(0,4x^{*}(1-x^{*})\right) \text{ when } \tau=1/2,\\
{}&\sqrt{n}\left\{S_{n}/n-(2x^{*}-1)\right\} \xlongrightarrow{D} N\left(0,4x^{*}(1-x^{*})/(2\tau-1)\right) \text{ when } \tau>1/2.\end{cases}
\end{equation}
\end{theorem}
Note that the uniqueness of the fixed point of $g$ in $[0,1]$ is guaranteed whenever one of \eqref{strict_decreasing_1}, \eqref{strict_convex_1} and \eqref{contraction_1} holds, along with $p\in(0,1)$. When $\tau<1/2$, the first criterion stated in \eqref{series_convergence_criterion_5} guarantees that \eqref{D4} holds.

\begin{corollary}\label{cor:main_1}
For the model described in \S\ref{subsec:with}, as long as the functions $f$ and $g$ satisfy the constraints stated in Theorem~\ref{thm:main_4}, and $k(n)=\Theta(n^{\alpha})$, where 
\begin{enumerate*}
\item $\alpha>\tau$ for $\tau\leqslant 1/2$ 
\item and $\alpha>1/2$ for $\tau>1/2$,
\end{enumerate*}
the conclusions of Theorem~\ref{thm:main_4} remain true.
\end{corollary}
Here, given any two sequences $\{a(n)\}$ and $\{b(n)\}$, we say that $a(n)=\Theta(b(n))$ if there exist constants $0<C_1<C_2$ such that $C_1 |b(n)| < |a(n)| < C_2 |b(n)|$ for all $n$ sufficiently large.

\subsection{Where the sampling scheme is without replacement}\label{subsec:without_results} Here, we consider the model described in \S\ref{subsec:without}, both when the sample size, $k(n)$, stays constant with $n$, and when $k(n)$ grows with $n$. The first two results, Theorem~\ref{thm:main_5} and Theorem~\ref{thm:main_6}, are concerned with almost sure convergence, while the latter two, Theorem~\ref{thm:main_7} and Theorem~\ref{thm:main_8}, are concerned with convergence in distribution.
\begin{theorem}\label{thm:main_5}
In the model described in \S\ref{subsec:without}, with $p\in(0,1)\setminus\{1/2\}$ and $k$ remaining constant as $n$ grows, if the function $H$, defined in \eqref{H_defn}, has a unique fixed point, say $x^{*}$, in $[-1,1]$, the sequence $\{S_{n}/n\}$ converges almost surely to $x^{*}$. Moreover, if one of \eqref{strict_decreasing}, \eqref{strict_convex} and \eqref{contraction} holds, the same conclusion as above is true.
\end{theorem}

\begin{theorem}\label{thm:main_6}
Consider the model described in \S\ref{subsec:without}, with $p\in(0,1)\setminus\{1/2\}$ and $k=k(n) \uparrow \infty$ as $n\rightarrow \infty$. We assume either that \eqref{D4} holds, or that one of \eqref{D1}, \eqref{D2}, \eqref{D3} holds and 
\begin{equation}\label{series_convergence_criterion_8}
\text{the series } \sum_{n}\exp\left\{\frac{k(n)^{2}}{n}\right\}\frac{k(n)^{2}}{n^{2}} \text{ converges}.
\end{equation}
We also assume that the function $g$ has a unique fixed point, say $x^{*}$, in $[0,1]$ (which, we note, is ensured if one of \eqref{strict_decreasing_1}, \eqref{strict_convex_1} and \eqref{contraction_1} holds). Then the sequence $\{S_{n}/n\}$ converges almost surely to $(2x^{*}-1)$.
\end{theorem}
Note that $k(n)=\Theta\left(n^{\alpha}\right)$ for any $0<\alpha<1/2$ allows \eqref{series_convergence_criterion_8} to be satisfied.

\begin{theorem}\label{thm:main_7}
In the model described in \S\ref{subsec:without}, with $k$ remaining constant as $n$ grows, suppose the function $H$, defined in \eqref{H_defn}, has a unique fixed point, say $x^{*}$, in $[-1,1]$. Letting $\tau=1-H'(x^{*})>0$, the same conclusion as that drawn in Theorem~\ref{thm:main_2} holds. Even more specifically, the conclusion remains true if one of \eqref{strict_decreasing}, \eqref{strict_convex} and \eqref{contraction} holds.
\end{theorem}

\begin{theorem}\label{thm:main_8}
Consider the model described in \S\ref{subsec:without}, with $p\in(0,1)\setminus\{1/2\}$ and the sample size $k(n)\uparrow\infty$ as $n\rightarrow\infty$. Assume that the function $f$, introduced via \eqref{X_{n+1}_distribution}, is in $\mathcal{C}^{2}[0,1]$, and the function $g$, defined in \eqref{g_defn}, has a unique fixed point $x^{*}\in[0,1]$. Letting $\tau=1-g'(x^{*})>0$, we assume that \eqref{series_convergence_criterion_5} holds when $\tau<1/2$, \eqref{series_convergence_criterion_6} holds when $\tau=1/2$, and \eqref{series_convergence_criterion_7} holds when $\tau>1/2$. In addition, we assume that \eqref{D4} holds for all values of $\tau$. Then the same conclusion as Theorem~\ref{thm:main_4} is true.
\end{theorem}

\section{Proofs of the results from \S\ref{sec:main_results}}\label{sec:all_proofs}
\subsection{Proofs of the results from \S\ref{subsec:results_with}}\label{subsec:proofs_with} We focus, in \S\ref{sec:all_proofs}, on the model described in \S\ref{subsec:with} (i.e.\ where the sampling scheme employed is one with replacement). 

\begin{proof}[Proof of Theorem~\ref{thm:main_1}]
Recalling, from \S\ref{subsec:with}, the definitions of the i.i.d.\ random variables, $U_{n,i}$ for all $i \in [k]$, we have
\begin{equation}
\Prob\left[U_{n,1}=i_{1},\ldots,U_{n,k}=i_{k}\right]=n^{-k} \quad \text{for all } i_{1},i_{2},\ldots,i_{k}\in[n].\nonumber
\end{equation}
As described in \S\ref{sec:intro}, a parallel may be drawn between our model and a generalized urn scheme as follows: we consider an urn that is empty to begin with, and that, at time-stamp $n$, contains a single ball which is of colour $(+1)$ with probability $q$ and of colour $(-1)$ with probability $(1-q)$. For $n \geqslant 2$, a sample of size $k$ is drawn, with replacement, from the existing balls in the urn, and we let $C_{n}^{+}$ (as defined in \eqref{C_{n}^{+}}) denote the number of $(+1)$-coloured balls appearing in this sample. The sampled balls are then returned to the urn, and a new ball, whose colour is random, and follows the same distribution as described in \eqref{X_{n+1}_distribution}, is added to the urn. Evidently, the urn contains a total of $n$ balls at time-stamp $n$, and evidently, $S_{n}$ equals the difference between the number of $(+1)$-coloured balls and the number of $(-1)$-coloured balls in the urn at time-stamp $n$, where $S_{n}$ is as defined in \eqref{S_{n}}. The number of $(+1)$-coloured balls in the urn, at time-stamp $n$, is then given by $(S_{n}+n)/2$, and the number of $(-1)$-coloured balls is given by $(n-S_{n})/2$. Thus, we may write:
\begin{equation}
\Prob\left[C_{n}^{+}=i\big|\mathcal{F}_{n}\right]=\frac{1}{n^{k}}{k \choose i}\left(\frac{S_{n}+n}{2}\right)^{i}\left(\frac{n-S_{n}}{2}\right)^{k-i}, \text{ for each } i \in [k]\cup\{0\},\label{C_{n}^{+}_cond_dist}
\end{equation}
where $\mathcal{F}_{n}$ is (as defined earlier) the $\sigma$-field comprising all information pertaining to the generalized urn process up to and including time-stamp $n$. From \eqref{C_{n}^{+}_cond_dist} and \eqref{X_{n+1}_distribution}, we conclude that 
\begin{align}
\Prob\left[S_{n+1}=S_{n}+1\big|\mathcal{F}_{n}\right]={}&\Prob\left[X_{n+1}=1\big|\mathcal{F}_{n}\right]=\sum_{i=0}^{k}\Prob\left[X_{n+1}=1\big|C_{n}^{+}=i,\mathcal{F}_{n}\right]\Prob\left[C_{n}^{+}=i\big|\mathcal{F}_{n}\right]\nonumber\\
={}&\frac{1}{2^{k}}\sum_{i=0}^{k}g\left(\frac{i}{k}\right){k \choose i}\left(1+\frac{S_{n}}{n}\right)^{i}\left(1-\frac{S_{n}}{n}\right)^{k-i},\label{cond_probab_constant_sample_size_with_replacement}
\end{align}
where the function $g$ is as defined in \eqref{g_defn}. From \eqref{cond_probab_constant_sample_size_with_replacement}, the fact that $X_{n+1}\in \{\pm1\}$, and \eqref{H_defn}, we have
\begin{equation}\label{cond_exp_constant_sample_size_with_replacement}
\E\left[X_{n+1}\big|\mathcal{F}_{n}\right]=2^{1-k}\sum_{i=0}^{k}g\left(\frac{i}{k}\right){k \choose i}\left(1+\frac{S_{n}}{n}\right)^{i}\left(1-\frac{S_{n}}{n}\right)^{k-i}-1=H\left(\frac{S_{n}}{n}\right).
\end{equation}

At this point, we come to what may be considered the principal technique implemented in the proofs of all of the results stated in \S\ref{sec:main_results}: that of representing our stochastic process as a stochastic approximation process. The specific representation to be considered depends on the set-up (i.e.\ whether the sampling scheme is with or without replacement, whether the sample size remains constant or grows with $n$, and whether we are interested in strong or weak convergence). For the proof of Theorem~\ref{thm:main_1}, we write, for $n \geqslant 1$,
\begin{align}
\frac{S_{n+1}}{n+1}={}&\frac{S_{n}+X_{n+1}}{n+1}=\frac{S_{n}}{n}+\frac{1}{n+1}\left\{X_{n+1}-\frac{S_{n}}{n}\right\}\nonumber\\
={}&\frac{S_{n}}{n}+\frac{1}{n+1}\left\{H\left(\frac{S_{n}}{n}\right)-\frac{S_{n}}{n}+X_{n+1}-H\left(\frac{S_{n}}{n}\right)\right\}=\frac{S_{n}}{n}+\gamma_{n}\left\{h\left(\frac{S_{n}}{n}\right)+\epsilon_{n+1}\right\},\label{sa_1}
\end{align}
where $H$ is as defined in \eqref{H_defn}, and we set
\begin{equation}
h(x)=H(x)-x \text{ for each } x \in [-1,1], \quad \gamma_{n}=\frac{1}{n+1} \quad \text{and} \quad \epsilon_{n+1}=X_{n+1}-H\left(\frac{S_{n}}{n}\right).\label{h_error_defns}
\end{equation}
In the literature on stochastic approximations (see, for instance, \cite{duflo2013random}, \cite{benaim2006dynamics} and \cite{borkar2008stochastic}), the function $h$ is referred to as the \emph{drift function}, the sequence $\{\gamma_{n}\}$ is referred to as the sequence of \emph{step sizes}, and the sequence $\{\epsilon_{n+1}\}$ is referred to as the sequence of \emph{errors}. The representation of the sequence $\left\{S_{n}/n\right\}$ in the form given by \eqref{sa_1} is referred to as a stochastic approximation process. From \eqref{cond_exp_constant_sample_size_with_replacement} and the definition of $\epsilon_{n+1}$ in \eqref{h_error_defns}, we have $\E\left[\epsilon_{n+1}\big|\mathcal{F}_{n}\right]=0$ for each $n \in \mathbb{N}$, which implies that $\{\epsilon_{n+1}\}$ forms a martingale difference sequence.

We now come to a well-known result from the literature on stochastic approximation processes, as this is instrumental in proving all our results pertaining to the almost sure convergence of the sequence $\{S_{n}/n\}$:
\begin{theorem}[Proposition 1.2.3 of \cite{duflo2013random}]\label{thm:gen_stoch_approx_a.s.}
Let $F$ be a continuous real-valued function such that, for a prespecified $\alpha \in \mathbb{R}$, there exists $x^{*} \in \mathbb{R}$ such that $F(x^{*})=\alpha$, and the following criteria are satisfied:
\begin{align}
{}&(F(x)-\alpha)(x-x^{*}) < 0 \text{ for all } x \neq x^{*},\label{eq:attractor_cond}\\
{}&|F(x)| \leqslant K(1+|x|) \text{ for all } x \in \mathbb{R}, \text{ for some constant } K > 0.\label{eq:uniform_bound_cond}
\end{align}
Let $\{\gamma_{n}\}$ be a sequence of positive reals decreasing towards $0$ such that $\sum_{n}\gamma_{n}$ diverges. Let $\{\epsilon_{n}\}$ be another sequence of reals such that $\sum_{n}\gamma_{n}\epsilon_{n+1}$ converges. Then the sequence $\{x_{n}\}$, defined iteratively by
\begin{equation}
    x_{n+1}=x_{n}+\gamma_{n}\{F(x_{n})-\alpha+\epsilon_{n+1}\} \text{ for each } n \in \mathbb{N},\label{general_stoch_approx_form}
\end{equation}
converges to $x^{*}$, for all initial values $x_{0}$.
\end{theorem}
The representation of the stochastic process $\{S_{n}/n\}$ in \eqref{sa_1} is of the same form as that of $\{x_{n}\}$ in \eqref{general_stoch_approx_form}, once we set $\alpha=0$, replace $x_{n}$ by $S_{n}/n$, and replace the function $F$ by the function $h$. Since $H$, defined in \eqref{H_defn}, is a polynomial on $[-1,1]$, and $h$ is as defined in \eqref{h_error_defns}, it is evident that $h$ is a continuous function. We now verify that \eqref{eq:uniform_bound_cond} holds for the function $h$, by expanding the polynomial that $h$ is, as follows:
\begin{align}\label{h_polynomial_form}
h(x)=\sum_{j=0}^{k}a_{j}x^{j}, \text{ where } a_{j}= 
  \begin{cases} 
   2^{1-k}\sum_{i=0}^{k}g\left(\frac{i}{k}\right){k \choose i}\sum_{\ell=0}^{j}(-1)^{j-\ell}{i \choose \ell}{k-i \choose j-\ell} & \text{for } j \in [k]\setminus\{1\}, \\
   2^{1-k}\sum_{i=0}^{k}g\left(\frac{i}{k}\right){k \choose i}(2i-k)-1 & \text{for } j=1,\\
   2^{1-k}\sum_{i=0}^{k}g\left(\frac{i}{k}\right){k \choose i}-1 & \text{for } j=0.
  \end{cases}
\end{align}
This then allows us to write
\begin{equation}
\left|h(x)\right|\leqslant \sum_{j=0}^{k}\left|a_{j}\right||x|^{j}\leqslant K(1+|x|) \text{ for all } x \in [-1,1],\quad \text{ where }
K=\max\left\{\left|a_{0}\right|,\sum_{j=1}^{k}\left|a_{j}\right|\right\}.\nonumber
\end{equation}
From \eqref{h_error_defns}, we have $\gamma_{n}=(n+1)^{-1}$, so that $\sum_{n}\gamma_{n}$ diverges. We recall from above that $\{\epsilon_{n+1}\}$ forms a martingale difference sequence, which, in turn, implies that so does $\{\gamma_{n}\epsilon_{n+1}\}$. We now prove the following lemma:
\begin{lemma}
The sum $\sum_{n}\gamma_{n}\epsilon_{n+1}$ converges almost surely.
\end{lemma}
\begin{proof}
Let us define $M_{0}=0$ and $M_{n}=\sum_{i=1}^{n}\gamma_{i}\epsilon_{i+1}$ for each $n \in \mathbb{N}$, so that $\{M_{n}\}$ forms a martingale with respect to the filtration $\{\mathcal{F}_{n+1}\}$ (note that $M_{n}$, by definition, is measurable with respect to $\mathcal{F}_{n+1}$). Using \eqref{h_error_defns}, \eqref{cond_exp_constant_sample_size_with_replacement} and the fact that $X_{i+1}\in\{\pm1\}\implies X_{i+1}^{2}=1$, the \emph{increasing process} $\{I_{n}\}$, associated with $\{M_{n}\}$, is given by:
\begin{align}
I_{n}={}&\sum_{i=1}^{n}\gamma_{i}^{2}\E\left[\epsilon_{i+1}^{2}\big|\mathcal{F}_{i}\right]
=\sum_{i=1}^{n}\gamma_{i}^{2}\E\left[\left\{X_{i+1}-H\left(\frac{S_{i}}{i}\right)\right\}^{2}\Bigg|\mathcal{F}_{i}\right]=\sum_{i=1}^{n}\frac{1}{(i+1)^{2}}\left[1-H^{2}\left(\frac{S_{i}}{i}\right)\right].\label{martingale_diff_second_moment}
\end{align} 
Here, the notation $H^{2}(x)$ indicates $(H(x))^{2}$. By definition of $H$ in \eqref{H_defn}, and the observation that $0\leqslant g(x)\leqslant 1$ for all $x \in [0,1]$ (which follows from \eqref{g_defn} and the fact that $f$ takes values in $[0,1]$), we have:
\begin{align}
-1 \leqslant H(x) \leqslant 2^{1-k}\sum_{i=0}^{k}{k \choose i}(1+x)^{i}(1-x)^{k-i}-1=1 \text{ for all } x \in [-1,1].\label{H_bounds}
\end{align}
Consequently, the final expression in \eqref{martingale_diff_second_moment} can be bounded above by $\sum_{i=1}^{n}(i+1)^{-2}$, which tells us that $I_{\infty}=\lim_{n \rightarrow \infty}I_{n}$ is finite almost surely. By Theorem 4.5.2.\ of \cite{durrett2019probability}, we conclude that the martingale $\{M_{n}\}$ converges to a finite limit almost surely. This completes the proof of our claim.
\end{proof}

The only criterion of Theorem~\ref{thm:gen_stoch_approx_a.s.} that is left to be verified for our drift function, $h$, is \eqref{eq:attractor_cond}, with $\alpha=0$. From \eqref{H_defn} and \eqref{H_bounds}, we note that $H$ is a continuous function (since it is a polynomial) mapping the convex, compact subset $[-1,1]$ of $\mathbb{R}$ to itself, so that by Brouwer's Fixed Point Theorem, $H$ must have at least one fixed point in $[-1,1]$. We now argue that neither $-1$ nor $1$ can be a fixed point of $H$ when $p\in(0,1)\setminus\{1/2\}$ (which is what we have assumed in \S\ref{sec:model}), which is equivalent to showing that $g(0)>0$ and $g(1)<1$, since $H(-1)=2g(0)-1$ and $H(1)=2g(1)-1$. Since $g(0)=(2p-1)f(0)+(1-p)$, it is immediate that $g(0)>0$ whenever $p\in(1/2,1)$, whereas when $p\in(0,1/2)$, we have $g(0)>0$ if and only if $f(0)<(1-2p)^{-1}(1-p)$, which is always true since $f(0)\leqslant 1$ and $(1-2p)^{-1}(1-p)>1$ for $p\in(0,1/2)$. Likewise, $g(1)=(2p-1)f(1)+(1-p)<1$ whenever $p\in(0,1/2)$, whereas when $p\in(1/2,1)$, we have $g(1)<1$ if and only if $f(1)<(2p-1)^{-1}p$, which is always true since $f(1)\leqslant 1$ and $(2p-1)^{-1}p>1$ for $p\in(1/2,1)$.

\begin{lemma}\label{lem:if_fixed_point_unique}
If $H$ has a unique fixed point, say $x^{*}$, in $(-1,1)$, then $h$ satisfies \eqref{eq:attractor_cond} with $\alpha=0$.
\end{lemma}
\begin{proof}
Suppose $H$ has a unique fixed point, $x^{*}$, in $(-1,1)$. By the argument outlined above, $H(-1)>-1$, so that the curve $y=H(x)$ travels from \emph{above} the line $y=x$, to \emph{beneath} the line $y=x$, at $x=x^{*}$. This implies that $H(x)>x$ for each $x\in[-1,x^{*})$, and $H(x)<x$ for each $x \in (x^{*},1]$, yielding the inequality in \eqref{eq:attractor_cond}.
\end{proof}

We are now in a position to explore a few sufficient conditions that ensure that the function $H$ has a unique fixed point in $(-1,1)$, so that by Lemma~\ref{lem:if_fixed_point_unique}, we may conclude that \eqref{eq:attractor_cond} holds:
\begin{enumerate}
\item When $H$ is strictly decreasing in the interval $(-1,1)$, the curve $y=H(x)$ is strictly decreasing whereas the line $y=x$ (in other words, the identity function) is strictly increasing, and consequently, they may intersect each other at most once.
\item When $H$ is either strictly convex or strictly concave throughout the interval $(-1,1)$, the curve $y=H(x)$ may intersect the line $y=x$ at most twice in $(-1,1)$. This, along with the fact that the curve $y=H(x)$ lies \emph{above} the line $y=x$ at $x=-1$ and \emph{beneath} the line $y=x$ at $x=+1$ (since $H(-1)>-1$ and $H(1)<1$, as argued in the paragraph before Lemma~\ref{lem:if_fixed_point_unique}), ensures that the curve $y=H(x)$ intersects the line $y=x$ precisely once in $(-1,1)$.
\item $H$ is a contraction over the interval $(-1,1)$: in this case, the uniqueness of the fixed point of $H$ in $(-1,1)$ is guaranteed by the Banach fixed point theorem.
\end{enumerate}
This brings us to the end of the proof of Theorem~\ref{thm:main_1}. 
\end{proof}

\begin{proof}[Proof of Proposition~\ref{prop:main_1}]
Differentiating the expression in \eqref{H_defn} with respect to $x$, and using \eqref{g_defn}, we obtain
\begin{align}
H'(x)
={}&2^{1-k}k\sum_{i=0}^{k-1}(2p-1)\left\{f\left(\frac{i+1}{k}\right)-f\left(\frac{i}{k}\right)\right\}{k-1 \choose i}(1+x)^{i}(1-x)^{k-i-1},\label{H_derivative}
\end{align}
from which it is evident that $H$ is strictly decreasing throughout $(-1,1)$ if one of the two criteria stated in Proposition~\ref{prop:main_1} holds. This, along with \eqref{strict_decreasing} of Theorem~\ref{thm:main_1}, completes the proof of Proposition~\ref{prop:main_1}.
\end{proof}
It is worthwhile to note that the criteria in \eqref{C1} and \eqref{C2} of Proposition~\ref{prop:main_1} are merely sufficient conditions -- weaker conditions could be proposed to ensure $H$ is strictly decreasing (for instance, by choosing $f$, as introduced via \eqref{X_{n+1}_distribution}, in such a way that the expression in \eqref{H_derivative} is strictly negative for all $x\in(-1,1)$).

\begin{proof}[Proof of Proposition~\ref{prop:main_2}]
Differentiating \eqref{H_derivative} with respect to $x$, and applying \eqref{g_defn}, we obtain:
\begin{align}
H''(x)
={}&2^{1-k}k(k-1)\sum_{j=0}^{k-2}(2p-1)\left\{f\left(\frac{j+2}{k}\right)-2f\left(\frac{j+1}{k}\right)+f\left(\frac{j}{k}\right)\right\}{k-2 \choose j}(1+x)^{j}(1-x)^{k-2-j}.\nonumber
\end{align}
If one of \eqref{strict_convexity_f_cond} and \eqref{strict_concavity_f_cond} is true, $H$ is either strictly convex throughout $(-1,1)$ or strictly concave throughout $(-1,1)$, thus satisfying \eqref{strict_convex}, and the conclusion follows from Theorem~\ref{thm:main_1}.
\end{proof}

\begin{proof}[Proof of Proposition~\ref{prop:main_3}]
Since $f(x) \in [0,1]$ for each $x \in [0,1]$, which enforces the inequalities $-1 \leqslant f((i+1)/k)-f(i/k)\leqslant 1$ for each $i \in [k-1]\cup\{0\}$, we obtain, via an application of the triangle inequality to \eqref{H_derivative}:
\begin{align}
\left|H'(x)\right|
\leqslant{}&2^{1-k}k|(2p-1)|\sum_{i=0}^{k-1}{k-1 \choose i}(1+x)^{i}(1-x)^{k-i-1}=k|(2p-1)| \text{ for all } x \in (-1,1),\nonumber
\end{align}
so that for $x, y \in (-1,1)$, applying the mean value theorem, we get, for some $\xi$ lying between $x$ and $y$:
\begin{equation}
\left|H(x)-H(y)\right|=\left|H'(\xi)\right||x-y|\leqslant k|(2p-1)||x-y|.\nonumber
\end{equation}
This shows that $H$ is a contraction on $(-1,1)$ when $p$ belongs to the interval $\left(1/2-1/(2k),1/2+1/(2k)\right)$, and by \eqref{contraction} of Theorem~\ref{thm:main_1}, the conclusion follows.
\end{proof}

\begin{proof}[Proof of Theorem~\ref{thm:main_2}]
The proof of Theorem~\ref{thm:main_2} relies on Theorem 2.2.12 of \cite{duflo2013random}, stated here:
\begin{theorem}[Theorem 2.2.12 of \cite{duflo2013random}]\label{thm:gen_stoch_approx_dist}
Suppose $\{A_{n}\}$ and $\{B_{n}\}$ denote a couple of sequences of random variables with values in $\mathbb{R}^{d}$, each of which is adapted to a filtration $\{\mathcal{G}_{n}\}$, and linked by the equation
\begin{equation}
A_{n+1}=A_{n}+n^{-1}B_{n+1},\label{eq:gen_stoch_approx_dist}
\end{equation}
such that
\begin{equation}
\E\left[B_{n+1}\big|\mathcal{G}_{n}\right]=F(A_{n}) \quad \text{and} \quad \E\left[\left\{B_{n+1}-F(A_{n})\right\}^{T}\left\{B_{n+1}-F(A_{n})\right\}\big|\mathcal{G}_{n}\right]=\Gamma(A_{n}).\label{eq:main_2_1}
\end{equation}
In addition, the following assumptions are made:
\begin{enumerate}[label=(H\arabic*), ref=H\arabic*]
\item \label{E1} the function $F$ is twice continuously differentiable, such that there exists a unique $x^{*}$ with
\begin{equation}
F(x^{*})=0 \quad \text{and} \quad \langle F(x),x-x^{*}\rangle<0 \quad \text{for all } x \neq x^{*},\label{unique_attractor_cond}
\end{equation}
and there exists $K>0$ such that for all $x \in \mathbb{R}^{d}$, we have
\begin{equation}
||F(x)||^{2}\leqslant K\left(1+||x||^{2}\right);\label{bound_1:thm_dist}
\end{equation}
\item \label{E2} the function $\Gamma$ is continuous in the neighbourhood of $x^{*}$ and satisfies, for some $K > 0$, 
\begin{equation}
\sigma^{2}(x)=\text{Trace}\left(\Gamma(x)\right)\leqslant K\left(1+||x||^{2}\right);\label{bound_2:thm_dist}
\end{equation}
\item \label{E3} there exists $a>0$ such that $\sup_{n}\E\left[||B_{n+1}-F(A_{n})||^{2+a}\big|\mathcal{G}_{n}\right]<\infty$.
\end{enumerate}
If the Jacobian matrix of $F$ at $x^{*}$ is equal to $-\tau I$ for some $\tau>0$, then, setting $\Gamma^{*}=\Gamma(x^{*})$, we have
\begin{align}
{}&\sqrt{n}\left(A_{n}-x^{*}\right)\xlongrightarrow{\text{D}}N\left(0,\Gamma^{*}(2\tau-1)^{-1}\right) \quad \text{when} \quad \tau>1/2,\nonumber\\
{}& \sqrt{\frac{n}{\ln n}}(A_{n}-x^{*})\xlongrightarrow{\text{D}}N(0,\Gamma^{*}) \quad \text{when} \quad \tau=1/2,\nonumber
\end{align}
and $n^{\tau}(A_{n}-x^{*})$ converges almost surely to a finite random variable when $\tau<1/2$.
\end{theorem}
Here, the notation $\langle\cdot,\cdot\rangle$ indicates the usual inner product, and $I$ indicates the $d\times d$ identity matrix.

We begin the proof of Theorem~\ref{thm:main_2} by setting, for each $n \in \mathbb{N}$ with $n\geqslant 3$,
\begin{equation}
\mathcal{G}_{n}=\mathcal{F}_{n-1},\quad A_{n}=\frac{S_{n-1}}{n-1} \quad \text{and} \quad B_{n}=h\left(\frac{S_{n-2}}{n-2}\right)+\epsilon_{n-1},\label{A_{n},B_{n}_defns}
\end{equation}
where we recall for the reader that $\mathcal{F}_{n}$ is the $\sigma$-field generated by the steps $X_{1}, X_{2}, \ldots, X_{n}$ taken by the walker up to and including time-stamp $n$, and $h$ and $\{\epsilon_{n}\}$ are as defined in \eqref{h_error_defns}. Incorporating the definitions from \eqref{A_{n},B_{n}_defns} into \eqref{eq:gen_stoch_approx_dist}, we obtain the relation given by \eqref{sa_1}. From \eqref{A_{n},B_{n}_defns}, \eqref{cond_exp_constant_sample_size_with_replacement} and \eqref{h_error_defns}, we have $\E\left[B_{n+1}\big|\mathcal{G}_{n}\right]=h(A_{n})$, so that the function $F$ of \eqref{eq:main_2_1} is replaced by the function $h$ in our set-up. Using \eqref{A_{n},B_{n}_defns}, \eqref{h_error_defns}, \eqref{cond_exp_constant_sample_size_with_replacement} and the fact that $X_{n}\in\{\pm 1\} \implies X_{n}^{2}=1$, we get, for $H$ as defined in \eqref{H_defn}:
\begin{align}
\E\left[\left\{B_{n+1}-F(A_{n})\right\}^{2}\big|\mathcal{G}_{n}\right]
=\Gamma\left(\frac{S_{n-1}}{n-1}\right), \quad \text{where} \quad \Gamma(x)=1-H^{2}(x) \text{ for } x\in[-1,1].\label{H(S_{n}/n)_bounded_by_1}
\end{align}
Since $H$ is a polynomial, $h$, defined in \eqref{h_error_defns}, is in $\mathcal{C}^{2}[-1,1]$. As argued in the paragraph preceding  Lemma~\ref{lem:if_fixed_point_unique} and Lemma~\ref{lem:if_fixed_point_unique} itself, if $H$ has a unique fixed point, $x^{*}$, in $[-1,1]$, and $p\in(0,1)\setminus\{1/2\}$, the inequality in \eqref{eq:attractor_cond}, with $\alpha=0$, is satisfied when the function $F$ is replaced by the function $h$ -- equivalently, \eqref{unique_attractor_cond} holds with $F$ replaced by $h$. From \eqref{h_polynomial_form} and the fact that $h$ is defined on $[-1,1]$, we have 
\begin{equation}
h^{2}(x)\leqslant K'(1+x^{2}), \quad \text{where} \quad K'=\max\left\{a_{0}^{2}+2|a_{0}|\sum_{i=1}^{k}|a_{i}|,\left(\sum_{i=1}^{k}|a_{i}|\right)^{2}\right\},\nonumber
\end{equation}
thus showing that \eqref{bound_1:thm_dist} holds. This completes the verification of all criteria stated in \eqref{E1}. Since $H$ is a polynomial, the function $\Gamma$, defined in \eqref{H(S_{n}/n)_bounded_by_1}, is continuous throughout $[-1,1]$, and \eqref{H_bounds} implies that $\Gamma(x)\leqslant 1$ for $x\in[-1,1]$, completing the verification of \eqref{E2}. Setting $a=2$, and using \eqref{A_{n},B_{n}_defns}, \eqref{h_error_defns}, \eqref{cond_exp_constant_sample_size_with_replacement} and the fact that $X_{n}\in\{\pm 1\} \implies X_{n}^{2}=1$, we have:
\begin{align}
{}&\E\left[||B_{n+1}-F(A_{n})||^{2+a}\big|\mathcal{G}_{n}\right]=\E\left[\left\{X_{n}-H\left(\frac{S_{n-1}}{n-1}\right)\right\}^{4}\big|\mathcal{F}_{n-1}\right]
=1+2H^{2}\left(\frac{S_{n-1}}{n-1}\right)-3H^{4}\left(\frac{S_{n-1}}{n-1}\right).\label{D3_computation}
\end{align}
\sloppy Applying \eqref{H_bounds} on \eqref{D3_computation} yields $\sup_{n}\E\left[||B_{n+1}-F(A_{n})||^{2+a}\big|\mathcal{G}_{n}\right]\leqslant 4/3$, thus completing the verification of \eqref{E3}. Finally, as argued while proving Lemma~\ref{lem:if_fixed_point_unique}, when $H$ has a unique fixed point, $x^{*}$, in $[-1,1]$, and $p\in(0,1)\setminus\{1/2\}$, the curve $y=H(x)$ travels from above $y=x$ to beneath $y=x$ at $x=x^{*}$, necessitating that the slope of $y=H(x)$ is strictly less than the slope of $y=x$ at $x=x^{*}$ -- equivalently, $H'(x^{*})<1$. Setting $\tau=-h'(x^{*})=1-H'(x^{*})$, and $\Gamma^{*}=1-H^{2}(x^{*})=1-{x^{*}}^{2}$, the rest of the conclusion of Theorem~\ref{thm:main_2} follows via Theorem~\ref{thm:gen_stoch_approx_dist}.
\end{proof}


\begin{proof}[Proof of Theorem~\ref{thm:main_3}]
Similar to the derivation of \eqref{C_{n}^{+}_cond_dist}, we deduce here that for $i\in[k(n)]\cup\{0\}$,
\begin{equation}
\Prob\left[C_{n}^{+}=i\big|\mathcal{F}_{n}\right]=\frac{1}{n^{k(n)}}{k(n) \choose i}\left(\frac{n+S_{n}}{2}\right)^{i}\left(\frac{n-S_{n}}{2}\right)^{k(n)-i}.\nonumber
\end{equation}
Incorporating this, we deduce, similar to how we derived \eqref{cond_probab_constant_sample_size_with_replacement} and \eqref{cond_exp_constant_sample_size_with_replacement}, that
\begin{equation}
\E\left[X_{n+1}\big|\mathcal{F}_{n}\right]=H_{n}\left(\frac{S_{n}}{n}\right) \quad \text{for each } n \in \mathbb{N}, \label{X_{n+1}_cond_exp_growing_sample_size}
\end{equation}
where we define the sequence of polynomials, $\{H_{n}\}$, on $[-1,1]$, as follows:
\begin{equation}
H_{n}(x)=2^{-k(n)}\sum_{i=0}^{k(n)}\left\{2g\left(\frac{i}{k(n)}\right)-1\right\}{k(n) \choose i}(1+x)^{i}(1-x)^{k(n)-i} \quad \text{for all }x\in[-1,1].\label{H_{n}_defn}
\end{equation} 
With a change of variable from $x$ to $y=1/2(1+x)$, we can rewrite the polynomial $H_{n}$ as
\begin{equation}
J_{n}(y)=\sum_{i=0}^{k(n)}\left\{2g\left(\frac{i}{k(n)}\right)-1\right\}{k(n) \choose i}y^{i}(1-y)^{k(n)-i} \text{ for } y \in [0,1].\nonumber
\end{equation}
If we set $\hat{g}(x)=2g(x)-1$ for $x \in [0,1]$, then $J_{n}$ equals the Bernstein polynomial of order $k(n)$ corresponding to $\hat{g}$. As long as $k(n)\uparrow \infty$ as $n\rightarrow \infty$, and $f$, defined in \eqref{X_{n+1}_distribution}, is continuous on $[0,1]$ (which ensures, via \eqref{g_defn}, that $\hat{g}\in\mathcal{C}^{0}[0,1]$), we have, by Weierstrass Approximation Theorem (Theorem 1.1.1 of \cite{lorentz2012bernstein}):
\begin{equation}
\lim_{n \rightarrow \infty}H_{n}(2y-1)=\lim_{n\rightarrow\infty}J_{n}(y)=\hat{g}(y) \text{ uniformly for all } y \in [0,1].\label{uniform_convergence}
\end{equation}
Performing a computation similar to that used to derive \eqref{sa_1}, we obtain
\begin{align}
\frac{S_{n+1}}{n+1}={}&\frac{S_{n}+X_{n+1}}{n+1}=\frac{S_{n}}{n}+\frac{1}{n+1}\left\{-\frac{S_{n}}{n}+X_{n+1}\right\}
=\frac{S_{n}}{n}+\frac{1}{n+1}\left\{\hat{h}\left(\frac{S_{n}}{n}\right)+\hat{\epsilon}_{n+1}+\delta_{n}\right\},\label{sa_2}
\end{align}
where we set, for $x\in[-1,1]$ and for $n\in \mathbb{N}$,
\begin{equation}
\hat{h}(x)=2g\left(\frac{1+x}{2}\right)-x-1, \quad \hat{\epsilon}_{n+1}=X_{n+1}-H_{n}\left(\frac{S_{n}}{n}\right) \quad \text{and} \quad \delta_{n}=H_{n}\left(\frac{S_{n}}{n}\right)-\hat{g}\left(\frac{1}{2}\left(1+\frac{S_{n}}{n}\right)\right).\label{hat_h_error_defns}
\end{equation}
The relation given by \eqref{sa_2} is of the same form as \eqref{general_stoch_approx_form} of Theorem~\ref{thm:gen_stoch_approx_a.s.}, with $x_{n}=S_{n}/n$ for each $n \in \mathbb{N}$, $\gamma_{n}=(n+1)^{-1}$ satisfying the criteria that $\gamma_{n}$ decreases to $0$ and $\sum_{n}\gamma_{n}$ diverges, the function $F$ replaced by the function $\hat{h}$ defined in \eqref{hat_h_error_defns}, $\alpha=0$, and $\epsilon_{n+1}=\hat{\epsilon}_{n+1}+\delta_{n}$, where $\hat{\epsilon}_{n+1}$ and $\delta_{n}$ are as defined in \eqref{hat_h_error_defns}.

We now verify the rest of the criteria stated in Theorem~\ref{thm:gen_stoch_approx_a.s.}, via Lemma~\ref{lem:series_gamma_{n}_epsilon_{n}_converges}, Lemma~\ref{lem:g_unique_fixed_point_consequence} and a short discussion in between.
\begin{lemma}\label{lem:series_gamma_{n}_epsilon_{n}_converges}
The sequence $\left\{\sum_{n=1}^{N}\gamma_{n}\epsilon_{n+1}\right\}$, where $\epsilon_{n+1}=\hat{\epsilon}_{n+1}+\delta_{n}$ (as defined in \eqref{hat_h_error_defns}), converges almost surely to a finite limit as $N \rightarrow \infty$.
\end{lemma}
\begin{proof}
Set $M_{0}=0$ and $M_{n}=\sum_{i=1}^{n}\gamma_{i}\hat{\epsilon}_{i+1}$ for each $n \in \mathbb{N}$. By \eqref{X_{n+1}_cond_exp_growing_sample_size} and \eqref{hat_h_error_defns}, we conclude that $\{M_{n}\}$ forms a martingale with respect to the filtration $\{\mathcal{F}_{n+1}\}$ (note that $M_{n}$ is measurable with respect to $\mathcal{F}_{n+1}$). A derivation similar to that of \eqref{martingale_diff_second_moment} yields, using \eqref{X_{n+1}_cond_exp_growing_sample_size}: 
\begin{equation}
I_{n}=\sum_{i=1}^{n}\gamma_{i}^{2}\E\left[\left\{X_{i+1}-H_{i}\left(\frac{S_{i}}{i}\right)\right\}^{2}\Big|\mathcal{F}_{i}\right]=\sum_{i=1}^{n}\frac{1}{(i+1)^{2}}\left\{1-H_{i}^{2}\left(\frac{S_{i}}{i}\right)\right\},\label{increasing_process_growing_sample_size}
\end{equation}
where $\left\{I_{n}\right\}$ is the increasing process associated with $\{M_{n}\}$. By \eqref{H_{n}_defn} and $0\leqslant g(x)\leqslant 1$ for all $x \in [0,1]$ (which follows from \eqref{g_defn} and the assumption that $f$ takes values in $[0,1]$), we have
\begin{align}
{}&-1 \leqslant H_{n}(x) \leqslant 1 \text{ for each } x \in [-1,1], \text{ for each } n \in \mathbb{N}.\label{H_{n}_bounded}
\end{align}
From \eqref{H_{n}_bounded} and \eqref{increasing_process_growing_sample_size}, we have $I_{n}\leqslant\sum_{i=1}^{n}(i+1)^{-2}$, so that $I_{\infty}=\lim_{n\rightarrow \infty}I_{n}<\infty$ almost surely. By Theorem 4.5.2.\ of \cite{durrett2019probability}, we conclude that the martingale $\{M_{n}\}$ converges to a finite limit almost surely. 

Next, we concern ourselves with the almost sure convergence of $\left\{\sum_{n=1}^{N}\gamma_{n}\delta_{n}\right\}$ as $N\rightarrow \infty$. In addition to \eqref{uniform_convergence}, we now need the rate at which the sequence $\{H_{n}(2y-1)\}$ converges to $\hat{g}(y)$ for \emph{all} $y \in [0,1]$.
\begin{enumerate}
\item Setting $r=0$ in Theorem 1 of \cite{guan2009iterated} (see, also, \cite{popoviciu1935approximation} and \cite{lorentz2012bernstein}) when $f\in\mathcal{C}^{0}[0,1]$ (which implies $\hat{g}\in\mathcal{C}^{0}[0,1]$), we obtain, for each $n\in\mathbb{N}$ with $k(n) \geqslant 2$, and for all $y \in [0,1]$:
\begin{align}
\left|H_{n}(2y-1)-\hat{g}(y)\right|\leqslant C_{0}\omega\left(\hat{g}, k(n)^{-1/2}\right)=2C_{0}\left|1-2p\right|\sup_{x,y\in[0,1],|x-y|<k(n)^{-1/2}}\left|f(x)-f(y)\right|,\label{f_continuous}
\end{align}
\sloppy where 
\begin{enumerate*}[label=(\roman*)]
\item $C_{0}$ is a constant, 
\item for any $\delta > 0$, the \emph{modulus of continuity} of $\hat{g}$ is defined as $\omega(\hat{g},\delta)=\sup\left\{\left|\hat{g}(x)-\hat{g}(y)\right|:x,y\in[0,1],|x-y|<\delta\right\}$,
\item and the second step in \eqref{f_continuous} follows from observing, via \eqref{g_defn}, that $\left|\hat{g}(x)-\hat{g}(y)\right|=2\left|(1-2p)\right|\left|f(x)-f(y)\right|$.
\end{enumerate*}
Since $k(n) \uparrow \infty$ as $n \rightarrow \infty$, there exists $N_{0}\in\mathbb{N}$ with $k(n)\geqslant 2$ for $n\geqslant N_{0}$. By \eqref{f_continuous} and \eqref{hat_h_error_defns}, for $N\geqslant N_{0}$, 
\begin{align}
\left|\sum_{n=1}^{N}\gamma_{n}\delta_{n}\right|\leqslant\sum_{n=1}^{N}\frac{1}{n+1}\left|\delta_{n}\right|= O\left(\sum_{n=1}^{N}\frac{1}{n+1}\sup\left\{\left|f(x)-f(y)\right|: x,y\in[0,1],|x-y|<k(n)^{-1/2}\right\}\right),\nonumber
\end{align}
which shows that $\left\{\sum_{n=1}^{N}\gamma_{n}\delta_{n}\right\}$ converges almost surely, as $N \rightarrow \infty$, when \eqref{series_convergence_criterion_1} holds. 

\item If $f$ is H\"{o}lder-continuous with exponent $\alpha$ and constant $L$, for $0< \alpha \leqslant 1$, we have, by \eqref{g_defn}, $\left|\hat{g}(x)-\hat{g}(y)\right|\leqslant 2L|(2p-1)||x-y|^{\alpha}$ for all $x, y \in [0,1]$, so that $\hat{g}$ is H\"{o}lder-continuous with exponent $\alpha$ and constant $2L|(2p-1)|$. By Theorem 1 of \cite{mathe1999approximation}, we conclude that, for $n\in\mathbb{N}$ and $y\in[0,1]$,
\begin{equation}
\left|H_{n}(2y-1)-\hat{g}(y)\right|\leqslant 2L|(2p-1)|\left(\frac{y(1-y)}{k(n)}\right)^{\alpha/2}\leqslant 2^{1-\alpha}L|(2p-1)|k(n)^{-\alpha/2},\nonumber
\end{equation}
which, in turn, yields
\begin{align}
\left|\sum_{n=1}^{N}\gamma_{n}\delta_{n}\right|\leqslant\sum_{n=1}^{N}\frac{1}{n+1}\left|\delta_{n}\right|\leqslant 2^{1-\alpha}L|(2p-1)|\sum_{n=1}^{N}(n+1)^{-1}k(n)^{-\alpha/2}.\nonumber
\end{align}
We now conclude that $\left\{\sum_{n=1}^{N}\gamma_{n}\delta_{n}\right\}$ converges almost surely as $N \rightarrow \infty$ when \eqref{series_convergence_criterion_2} holds.

\item If $f\in\mathcal{C}^{1}[0,1]$, then, by definition of $\hat{g}$ in terms of $g$ and \eqref{g_defn}, so is $\hat{g}$. Invoking, once again, Theorem 1 of \cite{guan2009iterated}, but this time with $r=1$, we obtain, for all $n\in\mathbb{N}$ with $k(n)\geqslant 2$ and all $y \in [0,1]$:
\begin{align}\label{f_differentiable}
{}&\left|H_{n}(2y-1)-\hat{g}(y)\right|\leqslant C_{1}k(n)^{-1/2}\omega^{(1)}\left(\hat{g},k(n)^{-1/2}\right)\nonumber\\
={}&2|(2p-1)|C_{1}k(n)^{-1/2}\sup\left\{\left|f'(x)-f'(y)\right|: x,y\in[0,1],|x-y|<k(n)^{-1/2}\right\},
\end{align}
where 
\begin{enumerate*}[label=(\roman*)]
\item $\omega^{(1)}(\hat{g},\delta)$, for $\delta>0$, is the modulus of continuity of the derivative $\hat{g}'$ of $\hat{g}$, defined as $\omega^{(1)}(\hat{g},\delta)=\sup\left\{\left|\hat{g}'(x)-\hat{g}'(y)\right|: x,y\in[0,1],|x-y|<\delta\right\}$,
\item and the second step of \eqref{f_differentiable} follows from observing that $\hat{g}'(x)=2(2p-1)f'(x)$.
\end{enumerate*}
Since $k(n)\uparrow \infty$ as $n\rightarrow \infty$, there exists $N_{0}$ with $k(n)\geqslant 2$ for $n\geqslant N_{0}$. For $N\geqslant N_{0}$, by \eqref{hat_h_error_defns} and \eqref{f_differentiable}:
\begin{align}
\left|\sum_{n=1}^{N}\gamma_{n}\delta_{n}\right|=O\left(\sum_{n=1}^{N}\frac{1}{n+1}k(n)^{-1/2}\sup\left\{\left|f'(x)-f'(y)\right|: x,y\in[0,1],|x-y|<k(n)^{-1/2}\right\}\right),\nonumber
\end{align}
which shows that $\left\{\sum_{n=1}^{N}\gamma_{n}\delta_{n}\right\}$ converges almost surely, as $N \rightarrow \infty$, when \eqref{series_convergence_criterion_3} holds. 

\item If $f\in\mathcal{C}^{2}[0,1]$, then so is $\hat{g}$, and by definition of $\hat{g}$ and \eqref{g_defn}, we have $\hat{g}''(x)=2(2p-1)f''(x)$ for $x\in[0,1]$. Since $f''$ is continuous throughout the compact interval $[0,1]$, the supremum $||f''||_{\infty}=\sup\{|f''(x)|:x\in[0,1]\}$ is finite. By Exercise 6.8 of \cite{powell1981approximation}, for $n\in\mathbb{N}$,
\begin{equation}
\left|H_{n}(2y-1)-\hat{g}(y)\right|\leqslant \frac{(2p-1)||f''||_{\infty}}{4k(n)} \text{ for all } y \in [0,1].\label{Bernstein_convergence_rate_twice_differentiable}
\end{equation}
Applying this bound, we have $\left|\sum_{n=1}^{N}\gamma_{n}\delta_{n}\right|=O\left(\sum_{n=1}^{N}(n+1)^{-1}k(n)^{-1}\right)$, which shows that $\left\{\sum_{n=1}^{N}\gamma_{n}\delta_{n}\right\}$ converges almost surely, as $N \rightarrow \infty$, when \eqref{series_convergence_criterion_4} holds.
\end{enumerate}  

We may now conclude that whenever one of \eqref{D1}, \eqref{D2}, \eqref{D3} and \eqref{D4} holds, the sequence $\left\{\sum_{n=1}^{N}\gamma_{n}\epsilon_{n+1}\right\}$, where $\epsilon_{n+1}=\hat{\epsilon}_{n+1}+\delta_{n}$, converges almost surely to a finite limit as $N \rightarrow \infty$. 
\end{proof}

Since $0\leqslant f(x) \leqslant 1$ for all $x\in[0,1]$, we have, by \eqref{g_defn}, the bounds $\min\{p,1-p\}\leqslant g(x) \leqslant \max\{p,1-p\}$ for all $x \in[0,1]$. These, along with \eqref{hat_h_error_defns}, implies $\left|\hat{h}(x)\right| \leqslant \left(1+2\max\{p,1-p\}\right)(1+|x|)$ for all $x\in[-1,1]$, thus verifying \eqref{eq:uniform_bound_cond} for the function $\hat{h}$. The only criterion left to verify for $\hat{h}$ is \eqref{eq:attractor_cond}. 

\begin{lemma}\label{lem:g_unique_fixed_point_consequence}
If $g$ has a unique fixed point, $x^{*}$, in $[0,1]$ and $p\in(0,1)$, then $\hat{h}$ has a unique root, $(2x^{*}-1)$, in $[-1,1]$, and $\hat{h}$ satisfies \eqref{eq:attractor_cond} with $\alpha=0$.
\end{lemma}
\begin{proof}
We begin by noting that if $x^{*}$ is a fixed point of $g$ in $[0,1]$, the corresponding root of $\hat{h}$ is $y^{*}=(2x^{*}-1)$). If $g$ has a unique fixed point, $x^{*}$, in $[0,1]$, such that $\{g(x)-x\}(x-x^{*})<0$ for all $x \in [0,1]\setminus\{x^{*}\}$, we have, for all $y \in [-1,1]$, writing $x=(y+1)/2$:
\begin{align}
\hat{h}(y)\left\{y-(2x^{*}-1)\right\}
=4\{g(x)-x\}(x-x^{*})<0,\nonumber
\end{align} 
thus verifying \eqref{eq:attractor_cond} for $\hat{h}$ with $\alpha=0$. Since $p\in(0,1)$, the bounds $\min\{p,1-p\}\leqslant g(x) \leqslant \max\{p,1-p\}$ for all $x \in[0,1]$ imply that $g(0)>0$ and $g(1)<1$. Hence, the curve $y=g(x)$ lies above $y=x$ for $x\in[0,x^{*})$, and beneath $y=x$ for $x\in(x^{*},1]$, ensuring $\{g(x)-x\}(x-x^{*})<0$ for $x\in[0,1]\setminus\{x^{*}\}$, and proving Lemma~\ref{lem:g_unique_fixed_point_consequence}.
\end{proof}

We now verify that $g$ has a unique fixed point in $[0,1]$ whenever one of \eqref{strict_decreasing_1}, \eqref{strict_convex_1} and \eqref{contraction_1} holds:
\begin{enumerate}
\item When $f$ is differentiable, so is $g$, with $g'(x)=(2p-1)f'(x)$. If $p>1/2$ and $f$ is strictly decreasing throughout $[0,1]$, or if $p<1/2$ and $f$ is strictly increasing throughout $[0,1]$, the function $g$ becomes strictly decreasing throughout $[0,1]$, and consequently, the curve $y=g(x)$ can intersect the line $y=x$ at most once (since $g(0)>0$ and $g(1)<1$, intersection must happen at least once).

\item When $f\in\mathcal{C}^{2}[0,1]$, so is $g$, with $g''(x)=(2p-1)f''(x)$. Since $p\in(0,1)\setminus\{1/2\}$, if $f$ is either strictly convex or strictly concave throughout $[0,1]$, then $g$ is either strictly convex or strictly concave throughout $[0,1]$ as well, so that the curve $y=g(x)$ can intersect the line $y=x$ at most twice. Since $g(0)>0$ and $g(1)<1$, the curve $y=g(x)$ lies above $y=x$ at $x=0$, and beneath it at $x=1$, which means that $y=g(x)$ intersects $y=x$ precisely once.

\item If the function $f$ is Lipschitz with Lipschitz constant $c$, then so is $g$, with $|g(x)-g(y)|=|(2p-1)||f(x)-f(y)|\leqslant c|2p-1||x-y|$. If, now, $c|(2p-1)|<1$, then $g$ becomes a contraction, and by the Banach fixed point theorem, we know that $g$ has a unique fixed point in $[0,1]$. \qedhere
\end{enumerate}  
\end{proof}

\begin{proof}[Proof of Theorem~\ref{thm:main_4}]
We base our proof of Theorem~\ref{thm:main_4} on the proof of Theorem 2.2.12.\ of \cite{duflo2013random}. Our stochastic approximation process for the set-up described in \S\ref{subsec:with}, with $k(n)\uparrow \infty$ as $n\rightarrow \infty$, is given by
\begin{equation}
\left\{\frac{S_{n+1}}{n+1}-(2x^{*}-1)\right\}=\left\{\frac{S_{n}}{n}-(2x^{*}-1)\right\}+\frac{1}{n+1}Y_{n+1}, \quad \text{with} \quad Y_{n+1}=X_{n+1}-\frac{S_{n}}{n},\label{sa_3}
\end{equation}
where, by \eqref{X_{n+1}_cond_exp_growing_sample_size} and \eqref{H_{n}_defn}, we have, for each $n\in\mathbb{N}$, 
\begin{equation}
\E\left[Y_{n+1}\big|\mathcal{F}_{n}\right]=h_{n}\left(\frac{S_{n}}{n}\right)=\hat{h}_{n}\left(\frac{S_{n}}{n}-(2x^{*}-1)\right), \text{ with } h_{n}(x)=H_{n}(x)-x,\ \hat{h}_{n}(x)=h_{n}(x+2x^{*}-1).\label{h_{n}_defn}
\end{equation}
Each of $h_{n}$ and $\hat{h}_{n}$, evidently, depends on $n$, unlike the set-up described in Theorem 2.2.12.\ of \cite{duflo2013random}, wherein $\E\left[Y_{n+1}\big|\mathcal{F}_{n}\right]$ equals some function of $S_{n}/n$ where the form of the function is independent of $n$. With $\tau=1-g'(x^{*})$, as defined in the statement of Theorem~\ref{thm:main_4}, we set $\gamma_{n}=(n+1)^{-1}$ and 
\begin{equation}
Z_{n}=Y_{n}-h_{n-1}\left(\frac{S_{n-1}}{n-1}\right), \quad \theta_{n}=h_{n}\left(\frac{S_{n}}{n}\right)+\tau\left\{\frac{S_{n}}{n}-(2x^{*}-1)\right\}, \quad \alpha_{n}=1-\tau\gamma_{n}, \quad \beta_{n}=\prod_{i=1}^{n}\alpha_{i},\label{many_definitions}
\end{equation}
for all $n\in\mathbb{N}$ with $n\geqslant 2$. Iterating \eqref{sa_3} and utilizing the notations introduced in \eqref{many_definitions}, we obtain:
\begin{align}
\left\{\frac{S_{n+1}}{n+1}-(2x^{*}-1)\right\}={}&\beta_{n}\left\{X_{1}-(2x^{*}-1)\right\}+\beta_{n}M_{n+1}+\beta_{n}R_{n},\label{iterated_recurrence}
\end{align}
where $X_{1}$, recall, is the first step of the walker, and for each $n\in\mathbb{N}$, we define: 
\begin{equation}
M_{n+1}=\sum_{i=2}^{n+1}\frac{\gamma_{i-1}}{\beta_{i-1}}Z_{i} \quad \text{and} \quad R_{n}=\sum_{i=1}^{n}\frac{\gamma_{i}}{\beta_{i}}\theta_{i}.\label{M_{n},R_{n}_defns}
\end{equation}
It thus suffices for us to understand the behaviour of the final expression in \eqref{iterated_recurrence} as $n\rightarrow\infty$.

As in Stages 2 and 3 of the proof of Theorem 2.2.12.\ of \cite{duflo2013random}, with $\gamma$ the Euler's constant and some $c>0$, 
\begin{align}
\lim_{n\rightarrow\infty}\beta_{n}(n+1)^{\tau}=\lim_{n\rightarrow\infty}\beta_{n}\exp\left\{\tau\sum_{k=1}^{n}\gamma_{k}\right\}\lim_{n\rightarrow\infty}\exp\left\{\tau\left(\ln(n+1)-\sum_{k=1}^{n}\gamma_{k}\right)\right\}=c' e^{\tau(1-\gamma)}=c.\label{beta_{n}_growth_rate}
\end{align}
From \eqref{beta_{n}_growth_rate}, it follows that the first term of \eqref{iterated_recurrence}, scaled the same way as in \eqref{eq:distributional_convergence_with_replacement_growing_size}, converges almost surely to the finite random variable $c\{X_{1}-(2x^{*}-1)\}$ when $\tau<1/2$, and to $0$ when $\tau\geqslant 1/2$. 

The analysis of the asymptotic behaviour of $\beta_{n}M_{n+1}$ begins by noting, by \eqref{h_{n}_defn} and the definition of $Z_{n}$ in \eqref{many_definitions}, that $\{M_{n}\}$ forms a martingale. The idea, now, is to conclude, by Theorem 4.5.2.\ of \cite{durrett2019probability}, that $\{M_{n}\}$ converges almost surely when $\tau<1/2$, and to apply the standard martingale central limit theorem to deduce the convergence in distribution of $\beta_{n}M_{n+1}$ when $\tau\geqslant 1/2$. Letting $\{I_{n}\}_{n\geqslant 2}$ denote the increasing process associated with $\{M_{n}\}$, we deduce, similar to \eqref{increasing_process_growing_sample_size} and using \eqref{beta_{n}_growth_rate} and \eqref{H_{n}_bounded}:
\begin{align}
I_{n}={}&
\sum_{k=2}^{n}\frac{\gamma_{k-1}^{2}}{\beta_{k-1}^{2}}\left\{1-H_{k-1}^{2}\left(\frac{S_{k-1}}{k-1}\right)\right\}=O\left(\sum_{k=1}^{n}k^{-2+2\tau}\right).\label{increasing_process_growing_sample_size_weak_convergence}
\end{align}
Evidently, when $\tau<1/2$, the second series in \eqref{increasing_process_growing_sample_size_weak_convergence} converges as $n\rightarrow \infty$, so that $I_{\infty}=\lim_{n\rightarrow\infty}I_{n}$ is finite almost surely. By Theorem 4.5.2.\ of \cite{durrett2019probability}, together with \eqref{beta_{n}_growth_rate}, we conclude that $\{(n+1)^{\tau}\beta_{n}M_{n+1}\}$ converges almost surely to a finite random variable when $\tau<1/2$. We now consider $\tau\geqslant 1/2$. Since all criteria stated in Theorem~\ref{thm:main_3} remain true in Theorem~\ref{thm:main_4}, $S_{n}/n\rightarrow (2x^{*}-1)$ almost surely. This, along with the uniform convergence deduced in \eqref{uniform_convergence}, the continuity of $\hat{g}^{2}$ (ensured by our assumption that $f\in\mathcal{C}^{2}[0,1]$), the fact that that $\hat{g}(x)\in[-1,1]$ for all $x\in[0,1]$ (since $0\leqslant g(x)\leqslant 1$ for all $x\in[0,1]$), and the uniform bounds on $\{H_{n}\}$ obtained in \eqref{H_{n}_bounded}, we conclude that
\begin{equation}
\lim_{n\rightarrow\infty}H_{n}^{2}\left(\frac{S_{n}}{n}\right)\rightarrow\hat{g}^{2}(x^{*}) \text{ as } n \rightarrow \infty, \text{ almost surely}.\label{H_{n}^{2}(S_{n}/n)_converges}
\end{equation}
From \eqref{beta_{n}_growth_rate}, we have
\begin{align}
{}&\lim_{n\rightarrow \infty}(\ln n)^{-1}\sum_{k=2}^{n}\frac{\gamma_{k-1}^{2}}{\beta_{k-1}^{2}}=c^{-2} \text{ when } \tau=1/2, \quad \lim_{n\rightarrow \infty}n^{1-2\tau}\sum_{k=2}^{n}\frac{\gamma_{k-1}^{2}}{\beta_{k-1}^{2}}=c^{-2}(2\tau-1)^{-1} \text{ when } \tau>1/2.\label{gamma_beta_ratio_series_tau_geq_1/2}
\end{align}
It is evident that the series $\sum_{n}\gamma_{k-1}^{2}/\beta_{k-1}^{2}$ diverges when $\tau\geqslant 1/2$. This, along with \eqref{increasing_process_growing_sample_size_weak_convergence}, \eqref{H_{n}^{2}(S_{n}/n)_converges} and Lemma 2.2.13.\ of \cite{duflo2013random}, tells us that for $\tau\geqslant 1/2$:
\begin{align}
\lim_{n\rightarrow\infty}\left(\sum_{k=2}^{n}\frac{\gamma_{k-1}^{2}}{\beta_{k-1}^{2}}\right)^{-1}I_{n}=\lim_{n\rightarrow\infty}\left(\sum_{k=2}^{n}\frac{\gamma_{k-1}^{2}}{\beta_{k-1}^{2}}\right)^{-1}\sum_{k=2}^{n}\frac{\gamma_{k-1}^{2}}{\beta_{k-1}^{2}}\left\{1-H_{k-1}^{2}\left(\frac{S_{k-1}}{k-1}\right)\right\}=1-\hat{g}^{2}(x^{*}).\label{limit_eq_1}
\end{align}
This verifies a criterion for the martingale central limit theorem (see Corollary 2.1.10.\ of \cite{duflo2013random}, or Theorem 2 of \cite{brown1971martingale}). Next, we verify the Lyapunov condition (see \S 2.1.4.\ of \cite{duflo2013random}) for $\tau=1/2$. A derivation akin to \eqref{D3_computation}, along with \eqref{beta_{n}_growth_rate} and \eqref{gamma_beta_ratio_series_tau_geq_1/2}, leads, for $\tau=1/2$, to 
\begin{align}
\left(\sum_{k=2}^{n}\frac{\gamma_{k-1}^{2}}{\beta_{k-1}^{2}}\right)^{-1}\sum_{k=2}^{n}\E[\left|M_{k}-M_{k-1}\right|^{4}\big|\mathcal{F}_{k-1}]
=O\left[\left(\sum_{k=2}^{n}\frac{\gamma_{k-1}^{2}}{\beta_{k-1}^{2}}\right)^{-1}\sum_{k=2}^{n}\left(\frac{\gamma_{k-1}}{\beta_{k-1}}\right)^{4}\right]=O\left[\frac{1}{\ln n}\sum_{k=2}^{n}\frac{1}{k^{2}}\right],\label{Lyapunov_cond}
\end{align}
which converges to $0$ as $n\rightarrow\infty$. We verify the Lindeberg condition (see assumption A2.\ of Corollary 2.1.10.\ of \cite{duflo2013random}, or Equation 2 of \cite{brown1971martingale}) when $\tau> 1/2$. Fixing any $\epsilon>0$, adopting a derivation similar to that of \eqref{Lyapunov_cond}, and using \eqref{beta_{n}_growth_rate} and \eqref{gamma_beta_ratio_series_tau_geq_1/2} , we have
\begin{align}
{}&\left(\sum_{k=2}^{n}\frac{\gamma_{k-1}^{2}}{\beta_{k-1}^{2}}\right)^{-1}\sum_{k=2}^{n}\E\left[\frac{\gamma_{k-1}^{2}}{\beta_{k-1}^{2}}Z_{k}^{2}\chi\left\{\left|\frac{\gamma_{k-1}Z_{k}}{\beta_{k-1}}\right|\geqslant\epsilon\sqrt{\sum_{k=2}^{n}\frac{\gamma_{k-1}^{2}}{\beta_{k-1}^{2}}}\right\}\Bigg|\mathcal{F}_{k-1}\right]\nonumber\\
={}&\epsilon^{-2}\left(\sum_{k=2}^{n}\frac{\gamma_{k-1}^{2}}{\beta_{k-1}^{2}}\right)^{-2}\sum_{k=2}^{n}\E\left[\left\{\epsilon\sqrt{\left(\sum_{k=2}^{n}\frac{\gamma_{k-1}^{2}}{\beta_{k-1}^{2}}\right)}\right\}^{2}\frac{\gamma_{k-1}^{2}}{\beta_{k-1}^{2}}Z_{k}^{2}\chi\left\{\left|\frac{\gamma_{k-1}Z_{k}}{\beta_{k-1}}\right|\geqslant\epsilon\sqrt{\sum_{k=2}^{n}\frac{\gamma_{k-1}^{2}}{\beta_{k-1}^{2}}}\right\}\Bigg|\mathcal{F}_{k-1}\right]\nonumber\\
\leqslant{}&\epsilon^{-2}\left(\sum_{k=2}^{n}\frac{\gamma_{k-1}^{2}}{\beta_{k-1}^{2}}\right)^{-2}\sum_{k=2}^{n}\E\left[\left|\frac{\gamma_{k-1}Z_{k}}{\beta_{k-1}}\right|^{4}\right]
=\Theta\left(n^{2-4\tau}\right)\left[\Theta(1)+\Theta\left(n^{4\tau-3}\right)\right]=o(1) \text{ when } \tau>1/2.\label{Lindeberg_cond}
\end{align}
Combining \eqref{limit_eq_1}, \eqref{Lyapunov_cond} and \eqref{Lindeberg_cond}, as well as \eqref{gamma_beta_ratio_series_tau_geq_1/2}, we deduce, via Corollary 2.1.10.\ of \cite{duflo2013random}, that
\begin{align}
{}&(\ln n)^{-1/2}M_{n} \xlongrightarrow{D} N\left(0,c^{-2}\left\{1-\hat{g}^{2}(x^{*})\right\}\right) \quad \text{when } \tau=1/2,\label{M_{n}_normality_tau=1/2}\\
{}&n^{1/2-\tau}M_{n} \xlongrightarrow{D} N\left(0,c^{-2}(2\tau-1)^{-1}\left\{1-\hat{g}^{2}(x^{*})\right\}\right) \quad \text{when } \tau>1/2.\label{M_{n}_normality_tau>1/2}
\end{align}

The final task is to prove that 
\begin{enumerate*}
\item $R_{n}$ converges almost surely when $\tau< 1/2$,
\item $(\ln n)^{-1/2}R_{n}$ converges in probability to $0$ when $\tau=1/2$,
\item and $n^{1/2-\tau}R_{n}$ converges in probability to $0$ when $\tau>1/2$,
\end{enumerate*}
where $\{R_{n}\}$ is as defined in \eqref{M_{n},R_{n}_defns}. From \eqref{hat_h_error_defns}, \eqref{uniform_convergence}, \eqref{h_{n}_defn} and the definition of $\hat{g}$ in the proof of Theorem~\ref{thm:main_3}, we know that $h_{n}\rightarrow\hat{h}$ uniformly on $[-1,1]$. From \eqref{M_{n},R_{n}_defns}, \eqref{many_definitions}, \eqref{h_{n}_defn} and \eqref{hat_h_error_defns}, we have
\begin{align}
|R_{n}|
\leqslant\sum_{i=1}^{n}\frac{\gamma_{i}}{\beta_{i}}\left|H_{i}\left(\frac{S_{i}}{i}\right)-\hat{g}\left(\frac{1}{2}\left(1+\frac{S_{i}}{i}\right)\right)\right|+\sum_{i=1}^{n}\frac{\gamma_{i}}{\beta_{i}}\left|\hat{h}\left(\frac{S_{i}}{i}\right)+\tau\left\{\frac{S_{i}}{i}-(2x^{*}-1)\right\}\right|.\label{R_{n}_split}
\end{align}
Since we have assumed $f\in\mathcal{C}^{2}[0,1]$, we may apply \eqref{Bernstein_convergence_rate_twice_differentiable}, and using \eqref{beta_{n}_growth_rate}, we obtain
\begin{align}
{}&(n+1)^{\tau}\beta_{n}\sum_{i=1}^{n}\frac{\gamma_{i}}{\beta_{i}}\left|H_{i}\left(\frac{S_{i}}{i}\right)-\hat{g}\left(\frac{1}{2}\left(1+\frac{S_{i}}{i}\right)\right)\right|=\Theta\left((n+1)^{\tau}\beta_{n}\right)\Theta\left(\sum_{i=1}^{n}(i+1)^{\tau-1}k(i)^{-1}\right),\nonumber
\end{align}
and for $\tau<1/2$, this converges because of \eqref{beta_{n}_growth_rate} and the first criterion in \eqref{series_convergence_criterion_5}. Likewise, for $\tau=1/2$, 
\begin{equation}
\sqrt{\frac{(n+1)}{\ln(n+1)}}\beta_{n}\sum_{i=1}^{n}\frac{\gamma_{i}}{\beta_{i}}\left|H_{i}\left(\frac{S_{i}}{i}\right)-\hat{g}\left(\frac{1}{2}\left(1+\frac{S_{i}}{i}\right)\right)\right| = \Theta\left(\sqrt{n+1}\beta_{n}\right)\Theta\left(\frac{1}{\sqrt{\ln(n+1)}}\sum_{i=1}^{n}\frac{k(i)^{-1}}{\sqrt{i+1}}\right),\nonumber
\end{equation}
and this converges to $0$ by \eqref{beta_{n}_growth_rate} and the first criterion of \eqref{series_convergence_criterion_6} when $\tau=1/2$. Finally, we have
\begin{equation}
\sqrt{n+1}\beta_{n}\sum_{i=1}^{n}\frac{\gamma_{i}}{\beta_{i}}\left|H_{i}\left(\frac{S_{i}}{i}\right)-\hat{g}\left(\frac{1}{2}\left(1+\frac{S_{i}}{i}\right)\right)\right|=\Theta\left((n+1)^{\tau}\beta_{n}\right)\Theta\left((n+1)^{1/2-\tau}\sum_{i=1}^{n}(i+1)^{\tau-1}k(i)^{-1}\right),\nonumber
\end{equation}
and this converges to $0$ by \eqref{beta_{n}_growth_rate} and the first criterion of \eqref{series_convergence_criterion_7} when $\tau>1/2$.

\sloppy We analyze the second series of \eqref{R_{n}_split}. Using Taylor expansion of $\hat{h}(x)$, for $x\in[-1,1]$, around $(2x^{*}-1)$, and noting, from \eqref{hat_h_error_defns}, that
\begin{enumerate*}
\item $\hat{h}\in\mathcal{C}^{2}[-1,1]$ since $f\in\mathcal{C}^{2}[0,1]$,
\item $\hat{h}(2x^{*}-1)=0$ since $x^{*}$ is a fixed point of $g$,
\item $\hat{h}'(2x^{*}-1)=-\tau$, with $\tau$ as defined in Theorem~\ref{thm:main_4},
\end{enumerate*}
we obtain, for $\xi$ between $(2x^{*}-1)$ and $x$,
\begin{align}
{}&\left|\hat{h}(x)+\tau\{x-(2x^{*}-1)\}\right|\leqslant \frac{1}{2}\left|\hat{h}''(\xi)\right|\{x-(2x^{*}-1)\}^{2}\leqslant M\{x-(2x^{*}-1)\}^{2},\nonumber
\end{align}
with $M=\sup\{|\hat{h}''(\zeta)|:\zeta\in[-1,1]\}<\infty$ as $[-1,1]$ is compact and $\hat{h}''$ is continuous. This allows us to write
\begin{align}
\sum_{i=1}^{n}\frac{\gamma_{i}}{\beta_{i}}\left|\hat{h}\left(\frac{S_{i}}{i}\right)+\tau\left\{\frac{S_{i}}{i}-(2x^{*}-1)\right\}\right|\leqslant M\sum_{i=1}^{n}\frac{\gamma_{i}}{\beta_{i}}\left\{\frac{S_{i}}{i}-(2x^{*}-1)\right\}^{2}.\label{R_{n}_second_series_bound}
\end{align}
By \eqref{sa_3}, \eqref{h_{n}_defn}, \eqref{H_{n}_bounded} (ensuring $-2\leqslant h_{n}(x)\leqslant 2$ for $x\in[-1,1]$ and $\E[\{X_{n+1}-H_{n}(S_{n}/n)\}^{2}|\mathcal{F}_{n}]\leqslant 1$):
\begin{align}
{}&\E\left[\left\{\frac{S_{n+1}}{n+1}-(2x^{*}-1)\right\}^{2}\Bigg|\mathcal{F}_{n}\right]
\leqslant\left[\frac{S_{n}}{n}-(2x^{*}-1)\right]^{2}+\frac{5}{(n+1)^{2}}+\frac{2}{n+1}\left[\frac{S_{n}}{n}-(2x^{*}-1)\right]h_{n}\left(\frac{S_{n}}{n}\right).\label{conditional_second_moment_bound_growing_sample_size_with_replacement}
\end{align}
To bound the third term of \eqref{conditional_second_moment_bound_growing_sample_size_with_replacement}, note that, given $0<\eta<\tau$, there exists $\epsilon>0$ such that $|x-(2x^{*}-1)|\leqslant\epsilon \implies \left|\hat{h}'(x)+\tau\right|\leqslant\eta$, so that for $x\in[2x^{*}-1,2x^{*}-1+\epsilon]$:
\begin{align}
{}& (-\tau-\eta)\{x-(2x^{*}-1)\}^{2}\leqslant \{x-(2x^{*}-1)\}\hat{h}(x) \leqslant (-\tau+\eta)\{x-(2x^{*}-1)\}^{2}.\nonumber
\end{align}
A similar bound holds for $x\in[2x^{*}-1-\epsilon,2x^{*}-1]$. These bounds, along with \eqref{Bernstein_convergence_rate_twice_differentiable}, yields
\begin{align}
{}&\{x-(2x^{*}-1)\}h_{n}(x)
\leqslant \frac{M'}{k(n)}+(-\tau+\eta)\{x-(2x^{*}-1)\}^{2} \quad \text{for } x\in[2x^{*}-1-\epsilon,2x^{*}-1+\epsilon],\label{x h_{n}(x) bound}
\end{align} 
\sloppy where $M'=(2p-1)||f''||_{\infty}/4$. Setting $T=\inf\left\{n\geqslant j: \left|S_{n}/n\right|>\epsilon\right\}$ for any fixed $j\in\mathbb{N}$ and $\epsilon>0$, setting $u_{n}=\E\left[\left\{S_{n}/n-(2x^{*}-1)\right\}^{2}\chi\{T>n\}\right]$ for $n>j$, and using \eqref{conditional_second_moment_bound_growing_sample_size_with_replacement} and \eqref{x h_{n}(x) bound}, we obtain:
\begin{align}
u_{n+1}
\leqslant{}&\left\{1-2(n+1)^{-1}(\tau-\eta)\right\}u_{n}+5(n+1)^{-2}+2M'(n+1)^{-1}k(n)^{-1}.\label{u_{n}_recurrence_bound}
\end{align}
Setting $\kappa=(\tau-\eta)$, $v_{n}=\{1-2\kappa(n+1)^{-1}\}$ and $w_{n}=v_{j}v_{j+1}\ldots v_{n}$ for $n\geqslant j$, and iterating \eqref{u_{n}_recurrence_bound}:
\begin{align}
u_{n+1}\leqslant w_{n}u_{j}+5w_{n}\sum_{i=j}^{n}w_{i}^{-1}(i+1)^{-2}+2M'w_{n}\sum_{i=j}^{n}w_{i}^{-1}k(i)^{-1}(i+1)^{-1}.\label{u_{n}_final_recurrence}
\end{align}
The first two terms on the right side of \eqref{u_{n}_final_recurrence} can be dealt with the same way as the corresponding terms in the proof of Theorem 2.2.12.\ of \cite{duflo2013random}. Similar to \eqref{beta_{n}_growth_rate}, we have $\lim_{n\rightarrow\infty}w_{n}(n+1)^{2\kappa}=c''\in(0,\infty)$ as $n \rightarrow \infty$, which, along with \eqref{beta_{n}_growth_rate}, yields, given any $\epsilon>0$, an $N\in\mathbb{N}$ such that $(c-\epsilon)(n+1)^{-\tau} \leqslant \beta_{n} \leqslant (c+\epsilon)(n+1)^{-\tau}$ and $(c''-\epsilon)(n+1)^{-2\kappa}\leqslant w_{n} \leqslant (c''+\epsilon)(n+1)^{-2\kappa}$ for $n\geqslant N$. Thus, for $n\geqslant N$:
\begin{align}
{}&\sum_{i=1}^{n}\frac{\gamma_{i}w_{i}}{\beta_{i}}\sum_{t=j}^{i-1}w_{t}^{-1}k(t)^{-1}(t+1)^{-1}
=O(1)+O\left[\sum_{i=N+1}^{n}(i+1)^{\tau-2\kappa-1}\left\{O(1)+O\left(\sum_{t=N}^{i-1}k(t)^{-1}(t+1)^{2\kappa-1}\right)\right\}\right]\nonumber\\
={}&O(1)+O\left[n^{\tau-2\kappa}\right]+O\left[\sum_{i=1}^{n}(i+1)^{\tau-1}\sum_{t=1}^{i-1}(t+1)^{-1}k(t)^{-1}\right].\label{third_term_u_{n}_recurrence_common_bound}
\end{align}
When $\tau>1/2$, choosing $\kappa>1/2$ ensures, together with the second criterion in \eqref{series_convergence_criterion_7} that the expression in \eqref{third_term_u_{n}_recurrence_common_bound}, when multiplied by $n^{1/2-\tau}$, is $o(1)$ as $n\rightarrow\infty$. When $\tau=1/2$, choosing $\kappa>\tau/2$, together with the second criterion in \eqref{series_convergence_criterion_6}, ensures that the expression in \eqref{third_term_u_{n}_recurrence_common_bound}, when multiplied by $(\ln n)^{-1/2}$, converges to $0$ as $n\rightarrow\infty$. For $\tau<1/2$, choosing $\kappa>\tau/2$, together with the second criterion in \eqref{series_convergence_criterion_5}, ensures that \eqref{third_term_u_{n}_recurrence_common_bound} converges. This completes the proof of Theorem~\ref{thm:main_4}.
\end{proof}

\begin{proof}[Proof of Corollary~\ref{cor:main_1}]
Since $k(n)=\Theta(n^{\alpha})$ with $\alpha>0$, hence $\sum_{i=1}^{n}(i+1)^{-1}k(i)^{-1}=\Theta\left(\sum_{i=1}^{n}i^{-1-\alpha}\right)$ converges, validating \eqref{D4}. For $\tau\leqslant 1/2$, we have
\begin{align}
{}&\sum_{i=1}^{n}(i+1)^{\tau-1}k(i)^{-1}=\Theta\left(\sum_{i=1}^{n}i^{\tau-1-\alpha}\right) \quad \text{and} \quad \sum_{i=1}^{n}(i+1)^{\tau-1}\sum_{t=1}^{i-1}(t+1)^{-1}k(t)^{-1}=\Theta\left(\sum_{i=1}^{n}i^{\tau-1-\alpha}\right),\nonumber
\end{align}
each of which converges since $\alpha>\tau$, ensuring \eqref{series_convergence_criterion_5} for $\tau<1/2$ and \eqref{series_convergence_criterion_6} for $\tau=1/2$. For $\tau>1/2$,
\begin{align}
{}&n^{1/2-\tau}\sum_{i=1}^{n}(i+1)^{\tau-1}k(i)^{-1}\Theta\left(n^{1/2-\alpha}\right) \quad \text{and} \quad n^{1/2-\tau}\sum_{i=1}^{n}(i+1)^{\tau-1}\sum_{t=1}^{i-1}(t+1)^{-1}k(t)^{-1}=\Theta\left(n^{1/2-\alpha}\right),\nonumber
\end{align}
each of which converges to $0$ since $\alpha>1/2$, ensuring \eqref{series_convergence_criterion_7} for $\kappa>1/2$. This completes the proof.
\end{proof}

\subsection{Proofs of the results from \S\ref{subsec:without_results}}\label{subsec:proofs_without}
We consider the model outlined in \S\ref{subsec:without}, where the sample from the past is drawn without replacement. We begin by observing, from \eqref{C_{n}^{+}}, that 
\begin{equation}
C_{n}^{+}\leqslant \min\left\{k,\sum_{i=1}^{n}\chi\{X_{i}=1\}\right\}=\min\left\{k,\sum_{i=1}^{n}\left(\frac{X_{i}+1}{2}\right)\right\}=\min\left\{k,\frac{S_{n}+n}{2}\right\}.\nonumber
\end{equation}
Likewise, the number of time-stamps sampled from $[n]$, during each of which the walker had taken a step of size $-1$, is bounded above by $\min\left\{k,(n-S_{n})/2\right\}$. Given $\mathcal{F}_{n}$, for the event $\left\{C_{n}^{+}=i\right\}$ to occur, where
\begin{equation}
i\leqslant \min\left\{k,\frac{S_{n}+n}{2}\right\} \quad \text{and} \quad k-i \leqslant \min\left\{k,\frac{n-S_{n}}{2}\right\},\label{without_replacement_plus_1_count_bounds}
\end{equation}
\begin{enumerate}
\item \label{step_1} we choose $i$ time-stamps from the set $\left\{j \in [n]: X_{j}=+1\right\}$ in ${(S_{n}+n)/2 \choose i}$ ways,
\item \label{step_2} we choose $(k-i)$ time-stamps from the set $\left\{j\in[n]: X_{j}=-1\right\}$ in ${(n-S_{n})/2\choose k-i}$ ways,
\item \label{step_3} we choose the $i$ indices, out of the $k$ indices of our sample, during which time-stamps from the set $\left\{j \in [n]: X_{j}=+1\right\}$ were chosen, in ${k \choose i}$ ways,
\item we assign the $i$ time-stamps chosen during \eqref{step_1} to the $i$ indices chosen during \eqref{step_3} in $i!$ ways, 
\item we assign the $(k-i)$ time-stamps chosen during \eqref{step_2} to the $(k-i)$ \emph{remaining} indices in our sample (after step \eqref{step_3}) in $(k-i)!$ ways.
\end{enumerate}
These observations, along with \eqref{without_replacement_distribution}, yields, for each $i\in\mathbb{N}_{0}$ satisfying \eqref{without_replacement_plus_1_count_bounds}:
\begin{equation}
\Prob\left[C_{n}^{+}=i\big|\mathcal{F}_{n}\right]={k \choose i}{(n+S_{n})/2 \choose i}{(n-S_{n})/2 \choose k-i}i!(k-i)!\left\{{n\choose k}k!\right\}^{-1}={(n+S_{n})/2 \choose i}{(n-S_{n})/2 \choose k-i}{n\choose k}^{-1}.\nonumber
\end{equation}
If we define, using the definition of generalized binomial coefficients, the function
\begin{equation}
F_{n}(x)=2{n\choose k}^{-1}\sum_{i=0}^{k}g\left(\frac{i}{k}\right){n(1+x)/2 \choose i}{n(1-x)/2 \choose k-i}-1 \quad \text{for } x\in[-1,1]\label{F_{n}_defn}
\end{equation} 
(note that $F_{n}$ can be defined for $k=k(n)$ as well), then we derive, the same way as we deduced \eqref{cond_probab_constant_sample_size_with_replacement}, that 
\begin{align}
{}&\E\left[X_{n+1}\big|\mathcal{F}_{n}\right]=F_{n}\left(\frac{S_{n}}{n}\right).\label{X_{n+1}_cond_exp_without_replacement}
\end{align}
The form of $F_{n}$ is dependent on $n$ -- our task, now, is to approximate this sequence of functions suitably by a single function independent of $n$, and estimate the error incurred thereby. To this end, we note that, given any two sequences $\{a_{j}\}_{0\leqslant j \leqslant i-1}$ and $\{b_{j}\}_{0\leqslant j \leqslant i-1}$ of non-negative real numbers, such that $\max\{a_{j},b_{j}\}\leqslant y$ for each $j\in\{0,1,\ldots,i-1\}$, we can write (iteratively, where the first iteration has been shown)
\begin{align}
\left|\prod_{j=0}^{i-1}a_{j}-\prod_{j=0}^{i-1}b_{j}\right|
=\left|a_{0}-b_{0}\right|\prod_{j=1}^{i-1}a_{j}+b_{0}\left|\prod_{j=1}^{i-1}a_{j}-\prod_{j=1}^{i-1}b_{j}\right|\leqslant\left|a_{0}-b_{0}\right|y^{i-1}+y\left|\prod_{j=1}^{i-1}a_{j}-\prod_{j=1}^{i-1}b_{j}\right|\leqslant y^{i-1}\sum_{j=0}^{i-1}\left|a_{j}-b_{j}\right|.\label{difference_of_products}
\end{align}
When the sample size $k$ remains constant, we approximate $F_{n}$ by the function $H$ defined in \eqref{H_defn}. Setting $y=(1+x)/2$ for $x\in(-1,1)$, using \eqref{difference_of_products}, and letting $Z$ follow Binomial$(k,y)$, we obtain:
\begin{align}
{}&\left|F_{n}(x)-H(x)\right|
\leqslant \Bigg\{2\sum_{i=0}^{k}g\left(\frac{i}{k}\right){k \choose i}\left[\prod_{j=0}^{i-1}\left(y-\frac{j}{n}\right)\prod_{j=0}^{k-i-1}\left(1-y-\frac{j}{n}\right)-y^{i}\prod_{j=0}^{i-1}\left(1-\frac{j}{n}\right)\prod_{j=0}^{k-i-1}\left(1-y-\frac{j}{n}\right)\right]\nonumber\\&+2\sum_{i=0}^{k}g\left(\frac{i}{k}\right){k \choose i}\left[y^{i}\prod_{j=0}^{i-1}\left(1-\frac{j}{n}\right)\prod_{j=0}^{k-i-1}\left(1-y-\frac{j}{n}\right)-y^{i}(1-y)^{k-i}\prod_{j=0}^{k-1}\left(1-\frac{j}{n}\right)\right]\Bigg\}\left\{\prod_{j=0}^{k-1}\left(1-\frac{j}{n}\right)\right\}^{-1}\nonumber\\
{}&\leqslant 2\Bigg\{\sum_{i=0}^{k}{k \choose i}(1-y)^{k-i+1}y^{i-1}\left(\sum_{j=0}^{i-1}\frac{j}{n}\right)+\sum_{i=0}^{k}{k \choose i}y^{i}(1-y)^{k-i-1}\left(\sum_{j=0}^{k-i-1}\frac{yj}{n}+\sum_{j=0}^{k-i-1}\frac{(1-y)i}{n}\right)\Bigg\}\left\{\prod_{j=0}^{k-1}\left(1-\frac{j}{n}\right)\right\}^{-1}\nonumber\\
{}&=2\left\{\prod_{j=0}^{k-1}\left(1-\frac{j}{n}\right)\right\}^{-1}\E\left[\frac{1-y}{y}\frac{Z^{2}}{2n}-\frac{1-y}{y}\frac{Z}{2n}+\frac{y}{1-y}\frac{(k-Z-1)(k-Z)}{2n}+\frac{Z(k-Z)}{n}\right]\nonumber\\
{}&=4\left\{\prod_{j=0}^{k-1}\left(1-\frac{j}{n}\right)\right\}^{-1}\frac{y(1-y)k(k-1)}{n}=O\left(\frac{1}{n}\right) \quad \text{for all } x\in(-1,1).\label{without_replacement_function_approximation_error}
\end{align}
Note that we have $\left|F_{n}(x)-H(x)\right|=0$ at each of $x=-1$ (corresponds to $y=0$) and $x=+1$ (corresponds to $y=1$). To approximate $F_{n}$ when $k(n)$ grows with $n$ and $f\in\mathcal{C}^{2}[0,1]$, we use (\cite{gonska2006peano}, [Theorem 2, \cite{tachev2008voronovskaja}], [Theorem 2, \cite{gonska2008remarks}]):
\begin{theorem}\label{thm:approx_F_{n}_k(n)_grows}
If $L:\mathcal{C}^{0}[0,1]\rightarrow\mathcal{C}^{0}[0,1]$ is a positive linear operator, then for any $f\in\mathcal{C}^{2}[0,1]$ and $x\in[0,1]$:
\begin{equation}
\left|L(f;x)-f(x)-\frac{1}{2}f''(x)L\left((e_{1}-x)^{2};x\right)\right|\leqslant \frac{1}{2}L\left((e_{1}-x)^{2};x\right)\tilde{\omega}\left(f'', \frac{1}{3}\sqrt{\frac{L\left((e_{1}-x)^{4};x\right)}{L\left((e_{1}-x)^{2};x\right)}}\right),\label{general_linear_operator}
\end{equation}
where $e_{n}(x)=x^{n}$ for each $x\in[0,1]$ and each $n\in\mathbb{N}$, $\omega(f,\cdot)$ is the modulus of continuity of the function $f$, and $\tilde{\omega}(f,\cdot)$ denotes the least concave majorant of $\omega(f,\cdot)$ given by 
\begin{equation}
\tilde{\omega}(f,\epsilon)=\sup_{0\leqslant x \leqslant \epsilon \leqslant y \leqslant 1,x\neq y}\frac{(\epsilon-x)\omega(f,y)+(y-\epsilon)\omega(f,x)}{y-x} \text{ for } 0\leqslant \epsilon \leqslant 1.\label{least_concave_majorant}
\end{equation}
\end{theorem}
\noindent We consider a sequence of linear operators, $\{L_{n}\}$, where, for any $n\in\mathbb{N}$, any $f\in\mathcal{C}^{2}[0,1]$ and $y\in[0,1]$, 
\begin{equation}
L_{n}(f;y)=\sum_{i=0}^{k(n)}f\left(\frac{i}{k(n)}\right){ny \choose i}{n(1-y) \choose k(n)-i}{n \choose k(n)}^{-1}.\label{L_{n}_defn}
\end{equation}
By Newton's binomial theorem, we have, for $y\in(0,1)$ and $n\in \mathbb{N}$ (the identity in \eqref{0th_moment} generalizes the fact that the probability mass function for a hypergeometric distribution must sum to $1$): 
\begin{equation}
\sum_{i=0}^{k(n)}{ny \choose i}{n(1-y) \choose k(n)-i}{n \choose k(n)}^{-1}=1.\label{0th_moment}
\end{equation}
Using \eqref{0th_moment} with $k(n)=k$, and $0\leqslant g(\cdot)\leqslant 1$, in \eqref{F_{n}_defn}, we get $-1\leqslant F_{n}(x)\leqslant 1$ for $x\in[-1,1]$. By \eqref{0th_moment}:
\begin{align}
{}&\sum_{i=0}^{k(n)}\frac{i}{k(n)}{ny \choose i}{n(1-y) \choose k(n)-i}{n \choose k(n)}^{-1}
=y\sum_{i=1}^{k(n)}{ny-1 \choose i-1}{n(1-y) \choose k(n)-i}{n-1 \choose k(n)-1}^{-1}=y.\label{1st_moment}
\end{align}
Next, using \eqref{0th_moment} and \eqref{1st_moment}, we deduce that 
\begin{align}
{}&\sum_{i=0}^{k(n)}\left(\frac{i}{k(n)}\right)^{2}{ny \choose i}{n(1-y) \choose k(n)-i}{n \choose k(n)}^{-1}
=\frac{y(ny-1)}{k(n)}\sum_{i=2}^{k(n)}{ny-2 \choose i-2}{n(1-y) \choose k(n)-i}{n-1 \choose k(n)-1}^{-1}\nonumber\\&+y\sum_{i=1}^{k(n)}\frac{1}{k(n)}{ny-1 \choose i-1}{n(1-y) \choose k(n)-i}{n-1 \choose k(n)-1}^{-1}
=\frac{y\{nyk(n)+n(1-y)-k(n)\}}{k(n)(n-1)}.\label{2nd_moment}
\end{align}
We are now in a position to compute the following, using \eqref{0th_moment}, \eqref{1st_moment} and \eqref{2nd_moment}:
\begin{align}
L_{n}\left((e_{1}-y)^{2};y\right)={}&\sum_{i=0}^{k(n)}\left(\frac{i}{k(n)}-y\right)^{2}{ny \choose i}{n(1-y) \choose k(n)-i}{n \choose k(n)}^{-1}
=\frac{y(1-y)\{n-k(n)\}}{k(n)(n-1)}.\label{2nd_moment_centered}
\end{align}
It suffices for us to use the crude bound $\tilde{\omega}(f'',\epsilon)\leqslant 2||f''||_{\infty}$ for any function $f\in\mathcal{C}^{2}[0,1]$ and any $0\leqslant \epsilon \leqslant 1$. Implementing this bound, \eqref{L_{n}_defn} and \eqref{2nd_moment_centered} in \eqref{general_linear_operator}, we obtain:
\begin{align}
{}&\left|\sum_{i=0}^{k(n)}f\left(\frac{i}{k(n)}\right){ny \choose i}{n(1-y) \choose k(n)-i}{n \choose k(n)}^{-1}-f(y)\right|\leqslant \frac{3||f''||_{\infty}}{2}\frac{y(1-y)\{n-k(n)\}}{k(n)(n-1)}=O\left(\frac{1}{k(n)}\right).\nonumber
\end{align}
We now come back to our specific set-up, and recalling $F_{n}$ defined in \eqref{F_{n}_defn}, we have, for all $x\in[-1,1]$:
\begin{align}
\left|F_{n}(x)-\hat{g}\left(\frac{x+1}{2}\right)\right|={}&\left|L_{n}\left(2g(\cdot)-1;\frac{1+x}{2}\right)-\left\{2g\left(\frac{x+1}{2}\right)-1\right\}\right|
=O\left(\frac{1}{k(n)}\right).\label{F_{n}_approximation_error}
\end{align}

When $k$ remains constant, the model in \S\ref{subsec:without} yields the following stochastic approximation process:
\begin{align}
\frac{S_{n+1}}{n+1}={}&\frac{S_{n}}{n}+\gamma_{n}\left\{\epsilon_{n+1}+\delta_{n}+h\left(\frac{S_{n}}{n}\right)\right\} \quad \text{for } n\geqslant k,\label{sa_4}
\end{align}
where $h$ is as defined in \eqref{h_error_defns}, and we set
\begin{align}
\gamma_{n}=\frac{1}{n+1}, \quad \epsilon_{n+1}=X_{n+1}-F_{n}\left(\frac{S_{n}}{n}\right) \quad \text{and} \quad \delta_{n}=F_{n}\left(\frac{S_{n}}{n}\right)-H\left(\frac{S_{n}}{n}\right).\label{epsilon_delta_without_replacement_constant_size}
\end{align}
When $k=k(n)$ grows with $n$, we have, instead of \eqref{sa_4}, for each $n\in\mathbb{N}$:
\begin{align}
\frac{S_{n+1}}{n+1}={}&\frac{S_{n}}{n}+\gamma_{n}\left\{\epsilon_{n+1}+\delta_{n,1}+\delta_{n,2}+\hat{h}\left(\frac{S_{n}}{n}\right)\right\},\label{sa_5}
\end{align}
where $\gamma_{n}$ and $\epsilon_{n+1}$ are as defined in \eqref{epsilon_delta_without_replacement_constant_size}, $\hat{h}$ is as defined in \eqref{hat_h_error_defns}, and with $H_{n}$ as defined in \eqref{H_{n}_defn} and $\hat{g}(\cdot)=2g(\cdot)-1$ as defined in the proof of Theorem~\ref{thm:main_3}, we set
\begin{equation}
\delta_{n,1}=F_{n}\left(\frac{S_{n}}{n}\right)-H_{n}\left(\frac{S_{n}}{n}\right) \quad \text{and} \quad \delta_{n,2}=H_{n}\left(\frac{S_{n}}{n}\right)-\hat{g}\left(\frac{1}{2}\left(1+\frac{S_{n}}{n}\right)\right).\label{delta_{n,1},delta_{n,2}}
\end{equation}

\begin{proof}[Proof of Theorem~\ref{thm:main_5}]
Here, we refer to the stochastic approximation in \eqref{sa_4}. The sequence $\{\gamma_{n}\}$ satisfies the requirements stated in Theorem~\ref{thm:gen_stoch_approx_a.s.}, and, as argued in the paragraph preceding Lemma~\ref{lem:if_fixed_point_unique}, as well as Lemma~\ref{lem:if_fixed_point_unique} itself, the function $h$ satisfies \eqref{eq:uniform_bound_cond} and \eqref{eq:attractor_cond} with $\alpha=0$ (the latter is true since $H$ has a unique fixed point in $[-1,1]$ and $p\in(0,1)$). In particular, any of \eqref{strict_decreasing}, \eqref{strict_convex} and \eqref{contraction} ensures that $H$ has a unique fixed point in $[-1,1]$. We, therefore, need only verify that $\sum_{n}\gamma_{n}\left(\epsilon_{n+1}+\delta_{n}\right)$ converges almost surely.   

The almost sure convergence of $\sum_{n}\gamma_{n}\epsilon_{n+1}$, with $\epsilon_{n+1}$ as defined in \eqref{epsilon_delta_without_replacement_constant_size}, is established the same way as in Lemma~\ref{lem:series_gamma_{n}_epsilon_{n}_converges}, via an application of Theorem 4.5.2.\ of \cite{durrett2019probability}. Next, from \eqref{epsilon_delta_without_replacement_constant_size} and \eqref{without_replacement_function_approximation_error}, we have:
\begin{align}
\sum_{i=k}^{n}\left|\gamma_{i}\delta_{i}\right|= \sum_{i=k}^{n}\frac{1}{i+1}\left|F_{i}\left(\frac{S_{i}}{i}\right)-H\left(\frac{S_{i}}{i}\right)\right|=O\left(\sum_{i=k}^{n}i^{-2}\right),\nonumber
\end{align}
thus proving that $\sum_{n}\gamma_{n}\delta_{n}$ converges almost surely. This concludes the proof of Theorem~\ref{thm:main_5}.
\end{proof}

\begin{proof}[Proof of Theorem~\ref{thm:main_6}] 
We now refer to the stochastic approximation in \eqref{sa_5}. Note that 
\begin{enumerate*}
\item the sequence $\{\gamma_{n}\}$ satisfies the requirements stated in Theorem~\ref{thm:gen_stoch_approx_a.s.}, 
\item the function $\hat{h}$ satisfies \eqref{eq:uniform_bound_cond}, as shown while proving Theorem~\ref{thm:main_3},  
\item as $g$ has a unique fixed point, $x^{*}$, in $[0,1]$, the function $\hat{h}$ satisfies \eqref{eq:attractor_cond}, with $\alpha=0$, by Lemma~\ref{lem:g_unique_fixed_point_consequence},
\item and the almost sure convergence of $\sum_{n}\gamma_{n}\epsilon_{n+1}$ is established the same way as that of $\sum_{n}\gamma_{n}\hat{\epsilon}_{n+1}$ in Lemma~\ref{lem:series_gamma_{n}_epsilon_{n}_converges}.
\end{enumerate*}
When one of \eqref{D1}, \eqref{D2} and \eqref{D3} holds, $\sum_{n}\gamma_{n}\delta_{n,2}$ converges by Lemma~\ref{lem:series_gamma_{n}_epsilon_{n}_converges}. Using the inequality $1-x/2\geqslant e^{-x}$ for $x\in[0,1.59]$ and noting that \eqref{series_convergence_criterion_8} ensures $k(n)=o(n)$, which, in turn, implies $2j/n<1.59$ for $j\in\{0,1,\ldots,k(n)-1\}$ for all $n$ large enough, we obtain
\begin{equation}
\left\{\prod_{j=0}^{k(n)-1}\left(1-\frac{j}{n}\right)\right\}^{-1}\leqslant \exp\left\{2\sum_{j=0}^{k(n)-1}\frac{j}{n}\right\}\leqslant \exp\left\{\frac{k(n)^{2}}{n}\right\}.\nonumber
\end{equation}
Using this observation, and emulating the derivation of \eqref{without_replacement_function_approximation_error} with $k$ replaced by $k(n)$ and $H$ by $H_{n}$, we get
\begin{align}
{}&\left|F_{n}(x)-H_{n}(x)\right|\leqslant 4\exp\left\{\frac{2k(n)^{2}}{n}\right\}\frac{y(1-y)k(n)\{k(n)-1\}}{n}=O\left(\exp\left\{\frac{k(n)^{2}}{n}\right\}\frac{k(n)^{2}}{n}\right),\nonumber
\end{align}
so that $\sum_{n}\gamma_{n}\delta_{n,1}$ converges by \eqref{series_convergence_criterion_8}. When \eqref{D4} holds, we have $\left|\delta_{n,1}+\delta_{n,2}\right|=O\left(k(n)^{-1}\right)$ by \eqref{F_{n}_approximation_error}, so that \eqref{series_convergence_criterion_4} ensures that $\sum_{n}\gamma_{n}\left(\delta_{n,1}+\delta_{n,2}\right)$ converges. By Theorem~\ref{thm:gen_stoch_approx_a.s.}, the proof of Theorem~\ref{thm:main_6} is complete.
\end{proof}

\begin{proof}[Proof of Theorem~\ref{thm:main_7}]
We follow the proofs of Theorem~\ref{thm:main_4} and Theorem 2.2.12.\ of \cite{duflo2013random}, with $k$ remaining constant. Here, the stochastic approximation representation of our process is
\begin{equation}
\left(\frac{S_{n+1}}{n+1}-x^{*}\right)=\left(\frac{S_{n}}{n}-x^{*}\right)+\frac{1}{n+1}Y_{n+1},\quad \text{with }Y_{n+1}\text{ as defined in \eqref{sa_3}.}\nonumber
\end{equation}
We set $f_{n}(x)=F_{n}(x)-x$ for $x\in[-1,1]$, with $F_{n}$ as defined in \eqref{F_{n}_defn}. For all $n\geqslant k$, we set
\begin{align}
Z_{n+1}=Y_{n+1}-f_{n}\left(\frac{S_{n}}{n}\right) \quad \text{and} \quad \theta_{n}=f_{n}\left(\frac{S_{n}}{n}\right)+\tau\left(\frac{S_{n}}{n}-x^{*}\right),\label{many_definitions_1}
\end{align}
and we define $\alpha_{n}$ and $\beta_{n}$ same as in \eqref{many_definitions}. Similar to \eqref{iterated_recurrence}, we can write, for $n\geqslant k$:
\begin{align}
\left(\frac{S_{n+1}}{n+1}-x^{*}\right)=\beta_{n}\left(\frac{S_{k}}{k}-x^{*}\right)+\beta_{n}M_{n+1}+\beta_{n}R_{n}, \quad\text{with}\quad M_{n+1}=\sum_{i=k}^{n}\frac{\gamma_{i}}{\beta_{i}}Z_{i+1} \text{ and } R_{n}=\sum_{i=k}^{n}\frac{\gamma_{i}}{\beta_{i}}\theta_{i}.\label{M_{n},R_{n}_defns_1}
\end{align}
As in the proof of Theorem~\ref{thm:main_4}, we make the following observations:
\begin{enumerate}[label=(S\arabic*), ref=S\arabic*]
\item \label{proof_step_1} The sequence $\{\beta_{n}\}$ satisfies \eqref{beta_{n}_growth_rate}, which, in turn, tells us that \eqref{gamma_beta_ratio_series_tau_geq_1/2} holds for $\tau\geqslant 1/2$, and that $\sum_{i=k}^{n}\gamma_{i}^{2}\beta_{i}^{-2}=O\left(\sum_{i=k}^{n}i^{2\tau-2}\right)$ converges when $\tau<1/2$.
\item The increasing process, $\{I_{n}\}$, associated with the martingale $\{M_{n}\}$, via a deduction similar to \eqref{increasing_process_growing_sample_size_weak_convergence}, equals 
\begin{equation}
I_{n}=\sum_{i=k}^{n}\frac{\gamma_{i}^{2}}{\beta_{i}^{2}}\left\{1-F_{i}^{2}\left(\frac{S_{i}}{i}\right)\right\} \leqslant \sum_{i=k}^{n}\frac{\gamma_{i}^{2}}{\beta_{i}^{2}}=O\left(\sum_{i=k}^{n}i^{2\tau-2}\right).\label{increasing_process_without_replacement}
\end{equation}
This converges for $\tau<1/2$, so that $\{M_{n}\}$ converges almost surely for $\tau<1/2$ by Theorem 4.5.2.\ of \cite{durrett2019probability}. This, along with \eqref{proof_step_1}, implies that $(n+1)^{\tau}\beta_{n}M_{n+1}$ converges almost surely for $\tau<1/2$. 

\item Next, when $k$ remains constant with $n$, we note that \eqref{without_replacement_function_approximation_error} implies the uniform convergence of $F_{n}$ to $H$. This, along with an argument similar to that for deducing \eqref{H_{n}^{2}(S_{n}/n)_converges}, and the fact that $S_{n}/n$ converges almost surely to $x^{*}$, yields $1-F_{n}^{2}\left(S_{n}/n\right) \rightarrow 1-H^{2}(x^{*})=1-{x^{*}}^{2}$ almost surely. This, along with \eqref{proof_step_1} and Lemma 2.2.13 of \cite{duflo2013random}, yields, the same way as we deduced \eqref{limit_eq_1}:
\begin{equation}
\lim_{n \rightarrow\infty}\left(\sum_{i=k}^{n}\gamma_{i}^{2}\beta_{i}^{-2}\right)^{-1}I_{n}=1-H^{2}(x^{*})=1-{x^{*}}^{2}\text{ almost surely when }\tau\geqslant 1/2.\nonumber
\end{equation}
That the Lyapunov condition, given by \eqref{Lyapunov_cond}, is true for $\tau=1/2$, and that the Lindeberg condition, given by \eqref{Lindeberg_cond}, is true for $\tau>1/2$, can be verified very similarly in this set-up. 

\item The observations made above yield, similar to \eqref{M_{n}_normality_tau=1/2} and \eqref{M_{n}_normality_tau>1/2}, via Corollary 2.1.10.\ of \cite{duflo2013random}:
\begin{align}
{}&\left(\ln n\right)^{-1/2}M_{n} \xlongrightarrow{D} N\left(0,c^{-2}\left(1-{x^{*}}^{2}\right)\right) \text{ when } \tau=1/2,\nonumber\\
{}& n^{1/2-\tau}M_{n} \xlongrightarrow{D} N\left(0,c^{-2}(2\tau-1)^{-1}\left(1-{x^{*}}^{2}\right)\right) \text{ when } \tau>1/2.\nonumber
\end{align}
\end{enumerate}

We now come to the asymptotics of $\{R_{n}\}$. From \eqref{M_{n},R_{n}_defns_1} and the definitions of $f_{n}$ and $h$, we have
\begin{align}
|R_{n}|\leqslant \sum_{i=k}^{n}\frac{\gamma_{i}}{\beta_{i}}\left|F_{i}\left(\frac{S_{i}}{i}\right)-H\left(\frac{S_{i}}{i}\right)\right|+\sum_{i=k}^{n}\frac{\gamma_{i}}{\beta_{i}}\left|h\left(\frac{S_{i}}{i}\right)+\tau\left(\frac{S_{i}}{i}-x^{*}\right)\right|.\label{R_{n}_split_1}
\end{align}
\sloppy By \eqref{without_replacement_function_approximation_error} and \eqref{beta_{n}_growth_rate}, the first series of \eqref{R_{n}_split_1} is $\sum_{i=k}^{n}\gamma_{i}\beta_{i}^{-1}|F_{i}(S_{i}/i)-H(S_{i}/i)|=O[\sum_{i=k}^{n}\gamma_{i}(i\beta_{i})^{-1}]=O(\sum_{i=k}^{n}i^{\tau-2})$. This observation, along with \eqref{beta_{n}_growth_rate}, yields, almost surely as $n\rightarrow\infty$, 
\begin{align}
{}&(n+1)^{\tau}\beta_{n}\sum_{i=k}^{n}\frac{\gamma_{i}}{\beta_{i}}\left|f_{i}\left(\frac{S_{i}}{i}\right)-h\left(\frac{S_{i}}{i}\right)\right| \text{ converges for } \tau<1/2,\nonumber\\
{}&\sqrt{\frac{n+1}{\ln(n+1)}}\beta_{n}\sum_{i=k}^{n}\frac{\gamma_{i}}{\beta_{i}}\left|f_{i}\left(\frac{S_{i}}{i}\right)-h\left(\frac{S_{i}}{i}\right)\right|=\left\{\sqrt{n+1}\beta_{n}\right\}\left(\ln(n+1)\right)^{-1/2}O\left(\sum_{i=k}^{n}i^{-3/2}\right) \rightarrow 0 \text{ for } \tau=1/2,\nonumber\\
{}&\sqrt{n+1}\beta_{n}\sum_{i=k}^{n}\frac{\gamma_{i}}{\beta_{i}}\left|f_{i}\left(\frac{S_{i}}{i}\right)-h\left(\frac{S_{i}}{i}\right)\right|=\left\{(n+1)^{\tau}\beta_{n}\right\}n^{1/2-\tau}O\left(\sum_{i=k}^{n}i^{\tau-2}\right)=O\left(n^{1/2-\tau}\right)+O\left(n^{-1/2}\right),\nonumber
\end{align}
which converges to $0$ when $\tau>1/2$. Since $h'(x^{*})=-\tau$, we deduce, the same way as \eqref{R_{n}_second_series_bound}, that 
\begin{align}
\sum_{i=k}^{n}\frac{\gamma_{i}}{\beta_{i}}\left|h\left(\frac{S_{i}}{i}\right)+\tau\left(\frac{S_{i}}{i}-x^{*}\right)\right|=O\left(\sum_{i=k}^{n}\frac{\gamma_{i}}{\beta_{i}}\left(\frac{S_{i}}{i}-x^{*}\right)^{2}\right).\nonumber
\end{align}
Similar to \eqref{conditional_second_moment_bound_growing_sample_size_with_replacement}, we can deduce that 
\begin{align}
\E\left[\left(\frac{S_{n+1}}{n+1}-x^{*}\right)^{2}\big|\mathcal{F}_{n}\right]\leqslant \left(\frac{S_{n}}{n}-x^{*}\right)^{2}+\frac{5}{(n+1)^{2}}+\frac{2}{n+1}\left(\frac{S_{n}}{n}-x^{*}\right)f_{n}\left(\frac{S_{n}}{n}\right).\label{R_{n}_second_series_bound_1}
\end{align}
Arguing as we did to deduce \eqref{x h_{n}(x) bound}, and by \eqref{without_replacement_function_approximation_error}, we have, given $0<\eta<\tau$, some $\epsilon>0$ such that
\begin{align}
{}&\left(x-x^{*}\right)f_{n}(x)\leqslant O(n^{-1})+(-\tau+\eta)(x-x^{*})^{2} \quad \text{for all }x\in[x^{*}-\epsilon,x^{*}+\epsilon].\label{R_{n}_second_series_bound_last_term_bound}
\end{align}
\sloppy Similar to \eqref{u_{n}_recurrence_bound}, setting $T=\{n\geqslant j:|S_{n}/n|>\epsilon\}$ for $j\geqslant k$ and $\epsilon>0$, and $u_{n}=\E[(S_{n}/n-x^{*})^{2}\chi\{T>n\}]$ for $n\geqslant j\geqslant k$, we obtain, from \eqref{R_{n}_second_series_bound_1} and \eqref{R_{n}_second_series_bound_last_term_bound}:
\begin{align}
u_{n+1}\leqslant \left\{1-2(n+1)^{-1}(\tau-\eta)\right\}u_{n}+5(n+1)^{-2}+O(n^{-2})=\left\{1-2(n+1)^{-1}(\tau-\eta)\right\}u_{n}+O(n^{-2}).\nonumber
\end{align}
This boils the set-up down to almost exactly what has been considered in the proof of Theorem 2.2.12.\ in \cite{duflo2013random}, and the rest of the conclusion follows similarly. This concludes the proof of Theorem~\ref{thm:main_7}.
\end{proof}

\begin{proof}[Proof of Theorem~\ref{thm:main_8}]
Once again, we follow the proof of Theorem~\ref{thm:main_4}. The stochastic approximation representation of the process is given by \eqref{sa_3}, and instead of \eqref{h_{n}_defn}, we have $\E\left[Y_{n+1}\big|\mathcal{F}_{n}\right]=f_{n}\left(S_{n}/n\right)$ by \eqref{X_{n+1}_cond_exp_without_replacement}, where $f_{n}(x)=F_{n}(x)-x$, with $F_{n}$ as defined in \eqref{F_{n}_defn}. We set 
\begin{equation}
Z_{n+1}=Y_{n+1}-f_{n}\left(\frac{S_{n}}{n}\right)=X_{n+1}-F_{n}\left(\frac{S_{n}}{n}\right) \quad \text{and} \quad \theta_{n}=f_{n}\left(\frac{S_{n}}{n}\right)+\tau\left\{\frac{S_{n}}{n}-\left(2x^{*}-1\right)\right\},\label{many_definitions_2}
\end{equation}
and we define $\alpha_{n}$ and $\beta_{n}$ as in \eqref{many_definitions}. Consequently, \eqref{iterated_recurrence} holds with $\{M_{n}\}$ and $\{R_{n}\}$ defined the same way as in \eqref{M_{n},R_{n}_defns}, but in terms of $Z_{i}$ and $\theta_{i}$ as defined in \eqref{many_definitions_2}. Moreover, \eqref{beta_{n}_growth_rate} holds, implying \eqref{gamma_beta_ratio_series_tau_geq_1/2}. Next, similar to \eqref{increasing_process_without_replacement}, we deduce that $I_{n}=\sum_{k=1}^{n-1}\gamma_{k}^{2}\beta_{k}^{-2}\left\{1-F_{k}^{2}\left(S_{k}/k\right)\right\}=O\left(\sum_{k=2}^{n}k^{2\tau-2}\right)$, which converges when $\tau<1/2$, so that $\{M_{n}\}$ converges almost surely to a finite limit by Theorem 4.5.2.\ of \cite{durrett2019probability}. This, along with \eqref{beta_{n}_growth_rate}, ensures that $(n+1)^{\tau}\beta_{n}M_{n+1}$ converges almost surely to a finite limit when $\tau<1/2$.

Since \eqref{gamma_beta_ratio_series_tau_geq_1/2} is true, $\sum_{k=1}^{n-1}\gamma_{k}^{2}\beta_{k}^{-2}$ diverges when $\tau\geqslant 1/2$. Since $-1\leqslant F_{n}(\cdot) \leqslant 1$ for each $n$, and $-1\leqslant \hat{g}(\cdot)=2g(\cdot)-1 \leqslant 1$, we conclude, via \eqref{F_{n}_approximation_error}, that $F_{n}^{2}(x)$ converges to ${\hat{g}}^{2}((x+1)/2)$ uniformly for all $x\in[-1,1]$. Since the conclusion of Theorem~\ref{thm:main_6} remains true here, $\{S_{n}/n\}$ converges almost surely to $(2x^{*}-1)$. Arguing the same way as we did to derive \eqref{H_{n}^{2}(S_{n}/n)_converges}, we get $F_{n}^{2}\left(S_{n}/n\right)\rightarrow \left\{2g(x^{*})-1\right\}^{2}=(2x^{*}-1)^{2}$, keeping in mind that $g(x^{*})=x^{*}$. The observations made in this paragraph, via Lemma 2.2.13.\ of \cite{duflo2013random}, yields 
\begin{align}
\left(\sum_{k=1}^{n-1}\gamma_{k}^{2}\beta_{k}^{-2}\right)^{-1}I_{n}\rightarrow 4x^{*}(1-x^{*}).\nonumber
\end{align}
The Lyapunov condition can be shown to hold for $\tau=1/2$ the same was as in \eqref{Lyapunov_cond}, and the Lindeberg condition can be shown to hold for $\tau>1/2$ the same way as in \eqref{Lindeberg_cond}, with $H_{k-1}$ replaced by $F_{k-1}$ everywhere (and using the bounds $-1\leqslant F_{n}(\cdot)\leqslant 1$ for each $n$). Combining everything, we conclude that \eqref{M_{n}_normality_tau=1/2} holds for $\tau=1/2$ and \eqref{M_{n}_normality_tau>1/2} holds for $\tau>1/2$.

From \eqref{M_{n},R_{n}_defns}, $\theta_{n}$ as defined in \eqref{many_definitions_2}, $f_{n}$ as defined above, $\hat{h}$ as defined in \eqref{hat_h_error_defns}, \eqref{F_{n}_approximation_error}, \eqref{beta_{n}_growth_rate} and \eqref{R_{n}_second_series_bound}:
\begin{align}
\left|R_{n}\right|\leqslant{}& 
\sum_{i=1}^{n}\frac{\gamma_{i}}{\beta_{i}}\left|F_{i}\left(\frac{S_{i}}{i}\right)-\left\{2g\left(\frac{1}{2}\left(1+\frac{S_{i}}{i}\right)\right)-1\right\}\right|+\sum_{i=1}^{n}\frac{\gamma_{i}}{\beta_{i}}\left|\hat{h}\left(\frac{S_{i}}{i}\right)+\tau\left\{\frac{S_{n}}{n}-(2x^{*}-1)\right\}\right|\nonumber\\
\leqslant{}&O\left(\sum_{i=1}^{n}(i+1)^{\tau-1}k(i)^{-1}\right)+M\sum_{i=1}^{n}\frac{\gamma_{i}}{\beta_{i}}\left\{\frac{S_{i}}{i}-(2x^{*}-1)\right\}^{2}.\label{R_{n}_split_2}
\end{align}
By the first criterion in \eqref{series_convergence_criterion_5} and \eqref{beta_{n}_growth_rate}, the first series in \eqref{R_{n}_split_2}, multiplied by $(n+1)^{\tau}\beta_{n}$, converges as $n\rightarrow\infty$ when $\tau<1/2$. By the first criterion in \eqref{series_convergence_criterion_6} and \eqref{beta_{n}_growth_rate}, the first series in \eqref{R_{n}_split_2}, multiplied by $\sqrt{(n+1)}\{\ln (n+1)\}^{-1/2}\beta_{n}$, converges to $0$ as $n\rightarrow\infty$ when $\tau=1/2$. Finally, by the first criterion in \eqref{series_convergence_criterion_7} and \eqref{beta_{n}_growth_rate}, the first series in \eqref{R_{n}_split_2}, multiplied by $\sqrt{(n+1)}\beta_{n}$, converges to $0$ as $n\rightarrow\infty$ when $\tau>1/2$.

Next, a derivation similar to that of \eqref{conditional_second_moment_bound_growing_sample_size_with_replacement} leads to
\begin{align}
{}&\E\left[\left\{\frac{S_{n+1}}{n+1}-(2x^{*}-1)\right\}^{2}\Bigg|\mathcal{F}_{n}\right]\leqslant\left\{\frac{S_{n}}{n}-(2x^{*}-1)\right\}^{2}+\frac{5}{(n+1)^{2}}+\frac{2}{n+1}\left\{\frac{S_{n}}{n}-(2x^{*}-1)\right\}f_{n}\left(\frac{S_{n}}{n}\right),\nonumber
\end{align}
and a derivation similar to that of \eqref{x h_{n}(x) bound}, using \eqref{F_{n}_approximation_error}, yields, given any $0 < \eta < \tau$, an $\epsilon>0$ such that 
\begin{equation}
\left\{x-(2x^{*}-1)\right\}f_{n}(x)\leqslant O\left(k(n)^{-1}\right)+\left(-\tau+\eta\right)\left\{x-(2x^{*}-1)\right\}^{2} \text{ for all } x\in[2x^{*}-1-\epsilon,2x^{*}-1+\epsilon].\nonumber
\end{equation}
Defining the stopping time $T$ and the sequence $\{u_{n}\}$ the same way as defined right before \eqref{x h_{n}(x) bound}, for $n>j\geqslant k$ and $\epsilon >0$, we see that an inequality similar to \eqref{u_{n}_recurrence_bound} holds, i.e.\ we have
\begin{align}
u_{n+1}\leqslant\left\{1-2(n+1)^{-1}(\tau-\eta)\right\}u_{n}+5(n+1)^{-2}+O\left((n+1)^{-1}k(n)^{-1}\right).\nonumber
\end{align}
Letting $\kappa$, $\{v_{n}\}$ and $\{w_{n}\}$ be defined right after \eqref{u_{n}_recurrence_bound}, we obtain a recurrence similar to \eqref{u_{n}_final_recurrence}:
\begin{align}
u_{n+1}\leqslant w_{n}u_{j}+5w_{n}\sum_{i=j}^{n}w_{i}^{-1}(i+1)^{-2}+O\left(w_{n}\sum_{i=j}^{n}w_{i}^{-1}k(i)^{-1}(i+1)^{-1}\right).\nonumber
\end{align}
The first two terms in the inequality above can be dealt with the same way as the corresponding terms in the proof of Theorem 2.2.12.\ in \cite{duflo2013random}, whereas the third term can be dealt with the same way as the corresponding term in \eqref{u_{n}_final_recurrence} -- in other words, the third term of the inequality above leads to the expression in \eqref{third_term_u_{n}_recurrence_common_bound}, and the rest of the conclusion follows the same way as the last part of the proof of Theorem~\ref{thm:main_4}.
\end{proof}

\section*{Acknowledgments}
Part of this research was carried out while AR was affiliated with the Indian Institute of Science Education and Research (IISER) Pune. Both authors acknowledge Krishanu Maulik for helpful discussions.

\section*{Funding}
AR acknowledges support from the IIM-K SGRP Research Grant (No.\ SGRP/2025-26/22). \section*{Data availability}
 This research is theoretical in nature and does not involve the analysis or generation of any datasets.

\section*{Competing Interests}
The authors declare that they have no competing interests.

\section*{Credit Statement}
Both authors contributed equally to the work.

\bibliography{ERW_bib}
\end{document}